\documentclass[12pt]{amsart}
  \usepackage{latexsym} 
  \usepackage[all]{xy}
  \usepackage{amsfonts} 
  \usepackage{amsthm} 
  \usepackage{amsmath} 
  \usepackage{amssymb}
  \usepackage{pifont}  
  \usepackage{enumerate}
  \usepackage{hyperref}  
\newcommand{\cC}{{C}}

  \def\sw#1{{\sb{(#1)}}} 
  \def\sco#1{{\sb{(#1)}}}

  \def\sut#1{{\sp{\langle #1\rangle}}}

  \def\<{{\langle}} 
  \def\>{{\rangle}}

  \def\eps{\varepsilon}

  \def\note#1{{}} 
 \def\can{{\rm \textsf{can}}}

  \def\note#1{} 
  \def\M{{\bf M}}

  \def\ev{{\rm ev}}

  \def\cC{{\mathcal C}} 
   
  \def\cD{{\mathcal D}}

  \def\lhom#1#2#3{{}{{\rm Hom}\sb{#1}(#2,#3)}} 
  \def\rhom#1#2#3{{{\rm Hom}\sb{#1}(#2,#3)}}

  \def\rend#1#2{{{\rm End}\sb{#1}(#2)}} 
    \def\Map{\mathrm{Map}}

  \def\Lend#1#2{{{\rm End}\sp{#1}(#2)}} 
  \def\Rend#1#2{{{\rm End}\sp{#1}(#2)}}

  \def\beq{\begin{equation}} 
  \def\eeq{\end{equation}} 
  \def\DC{{\Delta_\cC}} 
  \def \eC{{\eps_\cC}} 
  \def\DD{{\Delta_\cD}} 
  \def \eD{{\eps_\cD}}

  \def\im{{\rm Im}}

  \def\ot{{\otimes}} 
   
  \def\Hom{\mbox{\rm Hom}\,} 
  \def\Endd{\mbox{\rm End}\,}

  \newcommand{\Lra}{\Leftrightarrow} 
  \newcommand{\Ra}{\Rightarrow}

  \def\he{\hat{e}}
  \def\hc{\hat{c}}

    \def\hT{\widehat{T}}
    \def\hC{\widehat{C}}
    
    \def\hS{\widehat{S}}
    \def\hP{\widehat{P}}
    \def\hV{\widehat{V}}
       \def\hTheta{\widehat{\Theta}}
  \def \hlambda{\widehat{\lambda}}
    \def\ev{\mathrm{ev}}
   \def \hev{\widehat{\ev}}
    \def \hx{\widehat{x}}

        \def \he{\widehat{e}}
\def \hy{\widehat{y}}
        \def \hz{\widehat{z}}
    
    \def\balpha {\bar{\alpha}}
     \def\fS{{\mathfrak S}}
  \def\fR{{\mathfrak R}}

  \def\hrho{\hat{\varrho}}
  \def\hlambda{\hat{\lambda}}
 \def\hF{\widehat{F}}
 \def\otA{\ot}

  \newcounter{zlist} 
  \newenvironment{zlist}{\begin{list}{(\arabic{zlist})}{ 
  \usecounter{zlist}\leftmargin2.5em\labelwidth2em\labelsep0.5em 
  \topsep0.6ex
  \parsep0.3ex plus0.2ex minus0.1ex}}{\end{list}}

  \newcounter{blist} 
  \newenvironment{blist}{\begin{list}{(\alph{blist})}{ 
  \usecounter{blist}\leftmargin2.5em\labelwidth2em\labelsep0.5em 
  \topsep0.6ex 
  \parsep0.3ex plus0.2ex minus0.1ex}}{\end{list}} 

  \newcounter{rlist} 
  \newenvironment{rlist}{\begin{list}{(\roman{rlist})}{ 
  \usecounter{rlist}\leftmargin2.5em\labelwidth2em\labelsep0.5em 
  \topsep0.6ex 
  \parsep0.3ex plus0.2ex minus0.1ex}}{\end{list}}

\def\stac#1{\raise-.2cm\hbox{$\stackrel{\displaystyle\otimes}{\scriptscriptstyle{#1}}$}}

\def\cten#1{\raise-.2cm\hbox{$\stackrel{\displaystyle\widehat{\otimes}}
{\scriptscriptstyle{#1}}$}}

  \def\Label#1{\label{#1}\ifmmode\llap{[#1] }\else 
  \marginpar{\smash{\hbox{\tiny [#1]}}}\fi} 
  \def\Label{\label}

\swapnumbers
  \newtheorem{proposition}{Proposition}[section]
  \newtheorem{lemma}[proposition]{Lemma} 
  \newtheorem{corollary}[proposition]{Corollary} 
  \newtheorem{theorem}[proposition]{Theorem} 

  \theoremstyle{definition} 
  \newtheorem{definition}[proposition]{Definition}
    
  \newtheorem{example}[proposition]{Example} 
    \newtheorem{notation}[proposition]{Notation}

  \theoremstyle{remark} 
  \newtheorem{remark}[proposition]{Remark} 
  \newtheorem{remarks}[proposition]{Remarks}

  \newcounter{c} 
  \renewcommand{\[}{\setcounter{c}{1}$$} 
  \newcommand{\etyk}[1]{\vspace{-7.4mm}$$\begin{equation}\Label{#1} 
  \addtocounter{c}{1}} 
  \renewcommand{\]}{\ifnum \value{c}=1 $$\else \end{equation}\fi} 
\newcommand{\cov}{{\rm cov}}
\newcommand{\ocov}{\ol{\rm cov}}
\newcommand{\oX}{\ol{X}}
\newcommand{\oT}{\ol{T}}

\def\ot{\otimes}

\newcommand{\Cc}{\mathcal{C}}
\newcommand{\Dd}{\mathcal{D}}
\newcommand{\Ee}{\mathcal{E}}
\newcommand{\Ff}{\mathcal{F}}
\newcommand{\Gg}{\mathcal{G}}
\newcommand{\Hh}{\mathcal{H}}

\def\rightact{\hbox{$\leftharpoonup$}}
\def\leftact{\hbox{$\rightharpoonup$}}

\def\*C{{}^*\hspace*{-1pt}{\Cc}}

\def\text#1{{\rm {\rm #1}}}

\def\ol{\overline}

\begin{document} 

 \title{Bimodule herds} 
 \author{Tomasz Brzezi\'nski}
 \address{ Department of Mathematics, Swansea University, 
  Singleton Park, \newline\indent  Swansea SA2 8PP, U.K.} 
  \email{T.Brzezinski@swansea.ac.uk}   
  \author{Joost Vercruysse}
   \address{ Department of Mathematics, Swansea University, 
  Singleton Park, \newline\indent  Swansea SA2 8PP, U.K.
  \newline  and \newline\indent
Faculty of Engineering, Vrije Universiteit Brussel (VUB), \newline\indent B-1050 Brussels, Belgium}
\email{jvercruy@vub.ac.be}

    \date{June 2008} 
  \begin{abstract} 
  The notion of a {\em bimodule herd} is introduced and studied. A bimodule herd consists of a $B$-$A$ bimodule, its formal dual, called a {\em pen}, and a map, called a {\em shepherd}, which satisfies unitality and coassociativity conditions. It is shown that every bimodule herd gives rise to a pair of corings and coactions. If, in addition, a bimodule herd is {\em tame} i.e.\ it is faithfully flat and a progenerator, or if it is a progenerator and the underlying ring extensions are split, then these corings are associated to entwining structures; the bimodule herd is a Galois comodule of these corings. The notion of a {\em bicomodule coherd} is introduced as a formal dualisation of the definition of a bimodule herd. Every bicomodule coherd defines a pair of (non-unital) rings. It is shown that a tame $B$-$A$ bimodule herd defines a bicomodule coherd, and sufficient conditions for the derived rings to be isomorphic to $A$ and $B$ are discussed. The composition of bimodule herds via the tensor product is outlined. The notion of a bimodule herd is illustrated by the example of Galois co-objects of a commutative, faithfully flat Hopf algebra.
 \end{abstract} 
  \maketitle 

\section{Introduction}
In classical geometry a torsor or a principal homogenous space is a $G$-set $X$ on which the group $G$ acts transitively and freely. Equivalently, torsors can be defined as sets, termed {\em herds} (also called torsors), $X$ with  a structure mapping $X\times X\times X\to X$ satisfying some axioms; see \cite[page 170]{Pru:the},  \cite[page 202, footnote]{Bae:ein}.
In this formulation, the group $G$ is derived rather than given from the onset. This reconstruction of a group $G$ from the axioms of herds is standard and well-known. Perhaps less known is that, to a principal homogenous space one can also associate a groupoid, known as the {\em Ehresmann} or {\em gauge groupoid}; see \cite[Example~1.1.5]{Mac:gen}. If $G$ acts on $X$ from the right, the gauge groupoid acts from the left.

Both these points of view on herds and torsors together with the reconstructions of groups and groupoids are present in non-commutative geometry. On one hand non-commutative principal homogeneous spaces are represented by (faithfully flat) Hopf-Galois extensions or, more generally, coalgebra-Galois extensions. On the other hand the Hopf-algebra-free notion of a {\em quantum torsor} was introduced by Grunspan in  \cite{Gru:tor}. That a faithfully flat quantum torsor  is the same as a faithfully flat Galois object was observed in \cite{Sch:tor}. Independently, the notion of a {\em quantum heap} was proposed by \v Skoda \cite{Sko:hea}, and it has been shown that the category of copointed quantum heaps (i.e.\ quantum heaps with a specified character) is isomorphic to the category of Hopf algebras; this gives the way of reconstructing a Hopf algebra from a quantum heap. The gauge groupoid associated to a Hopf-Galois extension or the {\em Ehresmann-Schauenburg bialgebroid} was constructed and led to the development of bi-Galois theory in \cite{Sch:big}.  The Ehresmann coring for coalgebra-Galois extensions is described in \cite[pp.\ 392-3]{Brz:str}. The need to describe  Hopf-Galois extensions led to introduction of {\em $B$-torsors}  in  \cite{Sch:tor} (cf.\ \cite[Section~2.8]{Sch:Hop}), while the fully symmetric Hopf-bi-Galois theory necessitated studies of  {\em $A$-$B$ torsors}  in \cite[Chapter~5]{Hobst:PhD}. The most recent step in the approach to describing Galois-type extensions in terms of torsors was made in \cite{BohBrz:pre}, where (faithfully flat) bi-Galois objects for coring extensions were described in terms of (faithfully flat) {\em pre-torsors}.

In all these algebraic approaches to torsors, a non-commutative torsor or a non-commutative principal bundle or a Galois-type extension is assumed to be an algebra with additional structure. Yet, {\em Galois comodules} for corings \cite{ElKGom:com} have been recently shown to be an effective, general and unifying framework for the  Hopf- and coalgebra-Galois theory; see \cite{Brz:Gal}, \cite{Wis:Gal}, \cite{BohVer:Mor}. The aim of the present paper is to introduce and study {\em bimodule herds}, i.e.\ torsor-like objects that are not assumed to be algebras, and to show their close relationship to Galois comodules. By using the terminology which refers to the older notion of herds (or flocks) on sets\footnote{The term {\em herd} is the English translation of  the German {\em Schar} of Pr\"ufer \cite{Pru:the} and Baer \cite{Bae:ein}, advocated by Johnstone in \cite{Joh:her}.}, we want to stress that objects we study are no longer algebras (and hence are more general than previously studied torsors). At the same time we avoid a term {\em torsor} which might have been used in too many different contexts. On the other hand, as we will mention later and show in the last section of this paper, our notion in an abstract sense unifies the classical notion of herds or torsors with the non-commutative torsors.

We begin in Section~\ref{se.torsors} by defining what a bimodule herd is. The definition of a bimodule herd involves a bimodule and its dual. To keep the situation completely symmetric rather than considering one-sided duals with apriori no clear reason which side should be preferred, we consider a {\em formal dual} as given by evaluation and coevaluation maps in Definition~\ref{def.dual}. This is very reminiscent of Morita contexts and we discuss this relationship, by relating surjectivity of (co)evaluation maps with progenerator properties. Next we define a (tame) bimodule herd as a bimodule with a formal dual, called a {\em pen},  and a unital and coassociative structure map, called a {\em shepherd},  in Definition~\ref{def.torsor}. It is shown that in this very general setup one can associate two corings to a bimodule herd; see Corollary~\ref{cor.coring}. These corings can be understood as  ``gauge corings"{} associated to a bimodule herd. The first main result of Section~\ref{se.torsors}, Theorem~\ref{thm.entw}, reveals that in the tame case, each of these corings comes from an {\em entwining structure}. Thus, although a bimodule herd is defined by purely module-theoretic means and its definition makes no use of coalgebraic notions, tame bimodule herds are a source of corings and entwining structures. In fact Theorem~\ref{thm.torsor.galois} and Remark~\ref{rem.torsor.galois} show that tame bimodule herds can  be identified with finite Galois comodules of corings associated to entwining structures. The approach to Galois theory through bimodule herds is fully left-right symmetric and hence lays foundations for the theory of {\em bi-Galois comodules}.

In Section~\ref{se:co-torsors} we formally dualise the notion of a bimodule herd, and define bicomodule coherds. These are bicomodules of two corings with a counital and associative structure map. Although bicomodule coherds might seem at first as a mere dualisation of bimodule herds, the main reason for their introduction is revealed in Theorem~\ref{thm.cotor}, where it is shown how a coherd can be associated to a tame bimodule herd. 

Section~\ref{se:composition} is devoted to the description of ways in which two bimodule herds can be composed. It turns out that the composition via the tensor product is possible whenever the associated corings can form a smash coproduct. Various facets of bimodule herds discussed in this paper are illustrated in Section~\ref{se:co-object} by {\em Galois co-objects} for (commutative) Hopf algebras. As these objects are not algebras, even  in this simple case, the use of bimodule herds (rather than previously studied torsor and pre-torsor algebras) becomes inevitable. The composition of Galois co-objects is shown to coincide with the composition of corresponding bimodule herds. 
 
The paper is supplemented with an appendix, in which we describe a categorical formulation of bimodule herds. This is in-line with recent resurgence of interest in categorical aspects of module and comodule theory, and also indicates a categorical framework which unifies bimodule herds with standard geometric (or set-theoretic) herds.

\subsection*{Notation}
Throughout the paper, $\M_R$ denotes the category of right $R$-modules and right $R$-linear maps where $R$ is a unital associative ring. Similarly, we use notations ${_R\M}$ for the category of left $R$-modules and $\M^C$ for the category of right comodules of an $R$-coring $C$. For an $R$-coring $C$, $\Delta_C$ denotes the coproduct and $\eps_C$ denotes the counit. We refer to \cite{AndFul:rin} and \cite{BrzWis:cor} for comprehensive introductions. If $X$ is an object in a category, then  $X$ is also used to denote the identity morphism on $X$. 
Simple tensors often represent a finite sum of simple tensors.

In Sections~\ref{se.torsors}-\ref{se:composition}, $R$ and $S$ are unital associative rings, and 
$$
\alpha: R\to A, \qquad \beta: S\to B,
$$
are maps of associative unital rings. All $A$-, respectively $B$-modules are understood as $R$-, respectively $S$-modules via $\alpha$, respectively $\beta$.

\section{Bimodule herds}\label{se.torsors}
In this section the definition and fundamental properties of bimodule herds, including their relationship with Galois comodules, are given.

\subsubsection*{Formal duals} In the definition of a bimodule herd one needs to use a bimodule and its dual. We formalise this by introducing the notion of a {\em formal dual}. This might be well-known to ring and module theorists; the definition and basic properties of formal duals are included for completeness and for fixing the notation.
\begin{definition}\label{def.dual}
Let $T$ be a $B$-$A$ bimodule.  An $A$-$B$ bimodule $\hT$ is said to be a {\em formal dual} of $T$ if there exist an $A$-bimodule map 
$$
\ev : \hT \ot_S T\to A,
$$
and a $B$-bimodule map
$$
\hev: T\ot_R \hT \to B,
$$
rendering commutative the following diagrams:
\begin{equation}\label{diag.A}
\xymatrix{
T\ot_R \hT\ot_S T \ar[rr]^{\hev \ot_S T} \ar[d]_{T\ot_R\ev} && B\ot_ST \ar[d]\\
T\ot_R A \ar[rr] && T\, ,}
\end{equation}
\begin{equation}\label{diag.B}
\xymatrix{
\hT\ot_S T\ot_R \hT \ar[rr]^{\hT \ot_S \hev} \ar[d]_{\ev\ot_R\hT} &&\hT\ot_S B \ar[d]\\
 A\ot_R\hT  \ar[rr] && \hT\, .}
\end{equation}
Here the unlabeled arrows correspond to $A$- and $B$-multiplications on $T$ and $\hT$. 
\end{definition}

If $T$ and $\hT$ are bimodules forming  a Morita context, then $\hT$ is a formal dual of $T$. Also, if $B=\rend AT$, then  $T^* = \rhom A T A$ is a formal dual of $T$ (with $\ev$ the evaluation map, and $\hev$ the coevaluation map). 

\begin{lemma}\label{lemma.formal.dual} Let $T$ be a $B$-$A$ bimodule and $\hT$ an $A$-$B$ bimodule.
\begin{zlist}
\item Suppose that $\hT$ is a formal dual of $T$. Write $T^*$ for $\rhom ATA$ and ${}^*T$ for $\lhom BTB$.
\begin{blist}
\item The map $\hev$ is surjective if and only if $T$ is a finitely generated and projective right $A$-module, a faithful left $B$-module and the map
$$
\lambda : \hT \to T^*, \qquad \hx\mapsto [x\mapsto \ev(\hx \ot_Sx)],
$$
is an isomorphism of $A$-$S$ bimodules. If this happens, then $\hT$ is a generator as a right $B$-module, and the map
$$
T\ot_A \hT \to B, \qquad x\ot_A\hx \mapsto \hev(x\ot_R \hx),
$$
is an isomorphism of $B$-bimodules.

\item The map $\ev$ is surjective if and only if $T$ is a finitely generated and projective left $B$-module, a faithful right $A$-module and the map
$$
\hlambda : \hT \to {}^*T, \qquad \hx\mapsto [x\mapsto \hev(x \ot_R\hx)],
$$
is an isomorphism of $R$-$B$ bimodules. If this happens, then $\hT$ is a generator as a left $A$-module, and the map
$$
\hT\ot_B T \to A, \qquad \hx\ot_B x \mapsto \ev(\hx\ot_S x),
$$
is an isomorphism of $A$-bimodules.
\end{blist}
\item 
The functors $-\ot_BT:\M_B\to\M_A$ and $-\ot_A\hT:\M_A\to\M_B$ determine a pair of inverse equivalences (i.e.\ $T$ is a progenerator as right $A$-module and $B=\Endd_A(T)$) if and only if $\hT$ is a formal dual of $T$ such that both maps $\ev$ and $\hev$ are surjective.
\end{zlist}
\end{lemma}
\begin{proof}
(1)(a) Assume that the map $\hev$ is surjective and let $e_i\in T$, $\he_i\in \hT$ be such that $\hev\left(\sum_i e_i\ot_R \he_i\right) =1_B$. Then, for all $x\in T$,
$$
\sum_i e_i \lambda(\he_i)(x) = \sum_i e_i \ev(\he_i\ot_S x) = \sum_i \hev(e_i\ot_R \he_i) x =x,
$$
where the penultimate equality follows by \eqref{diag.A}. This means that $\{e_i\in T,\; \lambda(\he_i)\in T^*\}$ is a finite dual basis for the right $A$-module $T$. Using the above calculation, right $A$-linearity of $\ev$ and \eqref{diag.B} one easily verifies that the map
$$
\lambda^{-1} : T^* \to \hT, \qquad f\mapsto \sum_i f(e_i)\he_i,
$$
is the inverse of $\lambda$.  

Let $\ell : B\to \rend A T$ be defined by $b\mapsto [x\mapsto bx]$. Using diagrams \eqref{diag.A} and \eqref{diag.B} and the definition of $e_i$, $\he_i$ one can verify that
$$
\ell^{-1}:\rend AT \to B,\qquad s\mapsto \sum_i\hev(s(e_i)\ot_R \he_i),
$$
is the inverse of $\ell$. Thus, in particular, $\ell$ is injective, i.e.\ $T$ is a faithful left $B$-module. Combining $\ell$ with $\lambda$ and the fact that $T$ is a finitely generated and projective right $A$-module, the map $T\ot_A \hT \to B$, $x\ot_A \hx \mapsto \hev(x\ot_R \hx)$ is recovered as a chain of isomorphisms
$$
T\ot_A\hT \cong T\ot_A T^*\cong \rend AT \cong B.
$$
Finally, take any right $B$-module map $f: M\to N$ such that $\rhom B \hT f =0$. Then,
in the view of just proven isomorphism,
$$
\rhom A T {\rhom B \hT f} =0 \Lra  \rhom B {T\ot_A \hT} f  =0 \Lra \rhom B B f  =0 \Lra f =0 .
$$
Therefore,  $\hT$ is a generator of right $B$-modules. 

In the converse direction, assume that $T$ is a finitely generated and projective right $A$-module, $\lambda$ is an isomorphism and that the map $\ell : B\to \rend A T$ is a monomorphism (i.e.\ $T$ is a faithful left $B$-module). Let $\{e_i\in T, \; e^*_i\in T^*\}$ be a dual basis for right $A$-module $T$. The commutativity of diagram \eqref{diag.A} implies that the following diagram 
$$
\xymatrix{T\ot_R \hT \ar[rrrr] \ar[d]_\hev &&&& T\ot_A \hT \ar[d]^{T\ot_A\lambda} \\
B\ar[rr]^\ell && \rend A T && T\ot_A T^* \ar[ll]}
$$
 is commutative. The unmarked arrow in the top row is the canonical surjection, while the unmarked arrow in the bottom row is the standard coevaluation map. Apply the clockwise composition to $\sum_i e_i \ot_R \lambda^{-1}(e_i^*)$ to obtain an endomorphism of $T$,
 $$
 x\mapsto \sum_ i e_i\lambda\left( \lambda^{-1}\left(e_i^*\right)\right)(x) = \sum_ i e_ie_i^*(x) =x.
 $$
 Since $\ell$ is a monomorphism, the preimage of this map is the unit element of $B$, i.e.\ $1_B =\hev \left( \sum_i e_i \ot_R \lambda^{-1}\left(e_i^*\right)\right)$. Since $\hev$ is a $B$-bimodule map, the above equality implies that $\hev$ is a surjective map. 
 
 The assertions (1)(b) are proven in a symmetric way.

(2) Suppose first that $\hev$ and $\ev$ are surjective. Then we know by part (1), that $\hev$ and $\ev$ induce well-defined bijective maps $T\ot_A\hT\to B$ and $\hT\ot_BT\to A$. One can easily check that these induced maps form a Morita context between $A$ and $B$, which is strict by construction and hence the categories $\M_A$ and $\M_B$ are equivalent.

Conversely, if the functors $-\ot_BT$ and $-\ot_A\hT$ induce an equivalence between the categories $\M_B$ and $\M_A$, then this equivalence is induced by a Morita context $(B,A,T,\hT,\mu,\tau)$. By putting $\ev:\hT\ot_ST\to \hT\ot_BT\to A$, where the first map is the cannonical projection and the second map is the Morita map, and similarly $\hev:T\ot_R\hT\to T\ot_A\hT\to B$, we find that $\hT$ is a formal dual of $T$ such that $\ev$ and $\hev$ are surjective.
\end{proof}

\begin{corollary}\label{cor.faithflat}
Let $T$ be a $B$-$A$ bimodule with a formal dual $\hT$. 
If both $\ev$ and $\hev$ are surjective, then 
\begin{zlist}
\item $T$ is faithfully flat as left $S$-module if and only if $B$ is faithfully flat as left $S$-module.
\item $T$ is faithfully flat as right $R$-module if and only if $A$ is faithfully flat as right $R$-module.
\end{zlist}
\end{corollary}

\begin{proof}
(1) Suppose that $B$ is faithfully flat as a left $S$-module. Since $\ev$ and $\hev$ are surjective, $-\ot_BT:\M_B\to \M_A$ is an equivalence of categories (see Lemma \ref{lemma.formal.dual}), hence $T$ is faithfully flat as a left $B$-module. Therefore, ${_ST}\cong {_SB\ot_BT}$ is faithfully flat as well.

Conversely, the surjectivity of $\ev$ and $\hev$ implies that $-\ot_A\hT:\M_A\to\M_B$ is an equivalence of categories, hence $\hT$ is faithfully flat as 
a left $A$-module, so ${_SB}\cong {_ST\ot_A\hT}$ is faithfully flat as well.

Part (2) is proven in a symmetric way.
\end{proof}

\subsubsection*{$B$-$A$ herds and associated corings}
The main object of studies of this paper is given in the following

\begin{definition} \label{def.torsor} Let $T$ be a $B$-$A$ bimodule with a formal dual $\hT$. $T$ is called a {\em bimodule herd} or simply a {\em $B$-$A$ herd}
 provided that there exists an $S$-$R$ bimodule map $\gamma: T\to T\ot_R \hT\ot_S T$ rendering commutative the following diagrams
 \begin{equation}\label{eq.coev}
\xymatrix{
T \ar[rr]^{\gamma} \ar[d]_{\cong} && T\ot_R \hT\ot_S T \ar[d]^{\hev \ot_S T}\\
S\ot_S T \ar[rr]^{\beta \ot_S T} && B\ot_S T\, ,}
 \end{equation} 
 \begin{equation}\label{eq.ev}
\xymatrix{
T \ar[rr]^{\gamma} \ar[d]_{\cong} && T\ot_R \hT\ot_S T \ar[d]^{T \ot_R \ev}\\
T\ot_R R \ar[rr]^{T\ot_R \alpha} && T\ot_R A\, ,}
 \end{equation} 
 \begin{equation}\label{eq.ass}
\xymatrix{
T \ar[rrr]^-{\gamma} \ar[d]_{\gamma} &&& T\ot_R \hT\ot_S T \ar[d]^{\gamma\ot_R\hT \ot_S T}\\
T\ot_R \hT\ot_S T \ar[rrr]^-{T\ot_R\ot_S \hT\ot_S\gamma} &&& T\ot_R \hT\ot_S T\ot_R \hT\ot_S T\, .}
 \end{equation} 
 The map $\gamma$ is called the {\em shepherd}, and the formal dual $\hT$ is referred to as the {\em pen}.
 
The bimodule herd $(T,\gamma)$ is said to be {\em tame} provided $T$ satisfies conditions of Corollary~\ref{cor.faithflat}, i.e.\ the maps $\ev$ and $\hev$ are surjective and $T$ is faithfully flat as an $R$- and $S$-module.
\end{definition}

\begin{notation} \label{not.gamma}
Let $\hT$ be a formal dual of a $B$-$A$ bimodule $T$, and let $\gamma: T\to T\ot_R \hT \ot_S T$ be an $S$-$R$ bimodule map. The application of $\gamma$ to an element $x\in T$ is denoted by 
$$
\gamma (x) = x\sut 1\ot_R x\sut 2\ot_S x\sut 3,
$$
(summation implicit). Given  $\gamma$, define
$$
\gamma_A : \xymatrix{  T \ar[rr]^-\gamma && T\ot_R \hT\ot_S T \ar[rr] && T\ot_A \hT\ot_S T  },
$$
and
$$
\gamma_B : \xymatrix{  T \ar[rr]^-\gamma && T\ot_R \hT\ot_S T \ar[rr] && T\ot_R \hT\ot_B T  },
$$
where the second maps are the canonical surjections.
\end{notation}
\begin{lemma}\label{lemma.lin}
Let $\hT$ be a formal dual of a $B$-$A$ bimodule $T$, and let $\gamma: T\to T\ot_R \hT \ot_S T$ be an $S$-$R$ bimodule map. If $\gamma$ satisfies property \eqref{eq.coev}, then 
 $\gamma_A$ is a right $A$-module map. If $\gamma$ makes \eqref{eq.ev} commute,  then $\gamma_B$ is a left $B$-module map.
\end{lemma}
\begin{proof}
The $A$-linearity of $\gamma_A$ is proven by the explicit calculations that use  diagram \eqref{eq.coev} in the second and the last equalities, diagram \eqref{diag.B} in the third equality, and diagram \eqref{diag.A} in the fifth equality. For all $a\in A$ and $x\in T$,
\begin{eqnarray*}
\gamma_A(x)a &=& x\sut 1\ot_R x\sut 2\ot_S x\sut 3a\\
&=& x\sut 1\ot_A x\sut 2\hev\left(\left(x\sut 3a\right)\sut 1\ot_R \left(x\sut 3a\right)\sut 2\right)\ot_S \left(x\sut 3a\right)\sut 3\\
&=&  x\sut 1\ot_A \ev\left(x\sut 2 \ot _S \left(x\sut 3a\right)\sut 1\right)\left(x\sut 3a\right)\sut 2\ot_S \left(x\sut 3a\right)\sut 3\\
&=&  x\sut 1\ev\left(x\sut 2 \ot _S \left(x\sut 3a\right)\sut 1\right)\ot_A \left(x\sut 3a\right)\sut 2\ot_S \left(x\sut 3a\right)\sut 3\\
&=&  \hev\left(x\sut 1 \ot_R x\sut 2\right) \left(x\sut 3a\right)\sut 1\ot_A \left(x\sut 3a\right)\sut 2\ot_S \left(x\sut 3a\right)\sut 3 = \gamma_A(xa).
\end{eqnarray*}
The second statement is proven by symmetric arguments. 
\end{proof}

\begin{proposition}\label{prop.dual}
Let $\hT$ be a formal dual of a $B$-$A$ bimodule $T$, and let $\gamma: T\to T\ot_R \hT \ot_S T$ be an $S$-$R$ bimodule map. 
\begin{zlist}
\item Assume that the map $\gamma$ makes diagram \eqref{eq.coev} commute. For all right $A$-modules $N$, the map
$$
\Theta_N : \rhom A T N\ot_S T\to N\ot_A \hT \ot_ST, \quad f\ot_S x\mapsto \left(f\ot_A\hT\ot_S T\right)\left(\gamma_A\left(x\right)\right) ,
$$
is an isomorphism of right $A$-modules, natural in $N$. In particular, writing $T^* = \rhom ATA$,  
$$T^*\ot_S T\cong \hT\ot_S T, $$ 
as $A$-bimodules. Furthermore, the following diagram
$$
\xymatrix{ T^*\ot_S T \ar[dr]\ar[rr]^{\Theta_A}&&  \hT\ot_S T \ar[dl]^\ev \\
& A& ,}
$$
in which the unmarked arrow is the evaluation map, is commutative. Finally, if $T$ is completely faithful as a left $S$-module, then $\hT \cong T^*$.

\item Assume that the map $\gamma$ makes diagram \eqref{eq.ev} commute. For all left $B$-modules $N$, the map
$$
\hTheta_N : T\ot_R \lhom B T N \to T\ot_R \hT \ot_B N, \quad x\ot_R f\mapsto \left(T\ot_R\hT\ot_B f\right)\left(\gamma_B\left(x\right)\right) ,
$$
is an isomorphism of left $B$-modules, natural in $N$. In particular, writing ${}^*T = \lhom BTB$,  
$$T\ot_R{}^*T\cong T\ot_R \hT ,$$ 
as $B$-bimodules. Furthermore,  the following diagram
$$
\xymatrix{ T\ot_R {}^*T \ar[dr]\ar[rr]^{\hTheta_B}&&  T\ot_R \hT \ar[dl]^\hev \\
& B& ,}
$$
in which the unmarked arrow is the evaluation map, is commutative. Finally, if $T$ is  completely faithful as a right $R$-module, then $\hT \cong {}^*T$.
\end{zlist}
\end{proposition}
\begin{proof}
(1) By Lemma~\ref{lemma.lin}, $\Theta_N$ is a right $A$-module map. The inverse of $\Theta_N$ is given by
$$
\Theta_N^{-1}: N\ot_A \hT\ot_S T\to \rhom A T N \ot_S T, \qquad n\ot_A \hx\ot_S x \mapsto n\lambda(\hx)\ot_S x,
$$
where
$\lambda : \hT \to T^*$, $\hx\mapsto [x\mapsto \ev(\hx \ot_Sx)]$ is the map described in Lemma~\ref{lemma.formal.dual}. Indeed, first, for all $n\in N$, $\hx \in \hT$ and $x\in T$,
\begin{eqnarray*}
\Theta_N\circ\Theta_N^{-1}( n\ot_A\hx\ot_S x) &=& n\lambda(\hx)(x\sut 1)\ot_A x\sut 2\ot_S x\sut 3 \\
&=& n\ot_A \ev(\hx\ot_S x\sut 1)x\sut 2\ot_S x\sut 3 \\
&=& n\ot_A \hx\hev( x\sut 1\ot_R x\sut 2)\ot_S x\sut 3 =  n\ot_A\hx\ot_S x,
\end{eqnarray*}
where the third equality follows by \eqref{diag.B} and the final equality by \eqref{eq.coev}. Second, for all $f\in \rhom A T N$ and $x\in T$,
\begin{eqnarray*}
\Theta_N^{-1}\circ \Theta_N (f\ot_S x) &=& f(x\sut 1)\lambda(x\sut 2) \ot_S x\sut 3 \\
&=& f(x\sut 1) \ev(x\sut 2\ot_S -)\ot_S x\sut 3 =  f\left(x\sut 1 \ev\left( x\sut 2\ot_S -\right)\right)\ot_S x\sut 3\\
&=& f\left(\hev\left( x\sut 1 \ot_R x\sut 2\right)-\right)\ot_S x\sut 3 = f\ot_S x,
\end{eqnarray*}
where the third equality follows by the right $A$-linearity of $f$, the fourth equality is a consequence of \eqref{diag.A}, and the final equality follows by \eqref{eq.coev}.

The forms of $\Theta_N$ and $\Theta_N^{-1}$ imply immediately that these maps are natural in $N$. The commutativity of the triangle diagram follows by the following direct calculation, for all $f\in \rhom A T A$ and $x\in T$,
\begin{eqnarray*}
\ev\left(\Theta_A\left(f\ot_S x\right)\right) &=& \ev \left( f\left(x\sut 1\right)x\sut 2\ot_S x\sut 3\right) = f\left(x\sut 1\ev\left(x\sut 2\ot_S x\sut 3\right)\right)\\
&=&  f\left(\hev\left(x\sut 1\ot_R x\sut 2\right)x\sut 3\right) = f(x),
\end{eqnarray*}
where the second equality follows by the right $A$-linearity of $f$ and the left $A$-linearity of $\ev$, and the last equality is a consequence of \eqref{eq.coev}.

Finally, note that $\Theta_A^{-1} = \lambda\ot_S T$. The tensor functor $-\ot_S T$ of a completely faithful module reflects exact sequences (see \cite[page~233]{AndFul:rin}), hence it also reflects isomorphisms. Thus, if $T$ is a completely faithful left $S$-module, $\lambda$ is the required isomorphism. 

The statement (2) is proven by symmetric arguments.
\end{proof}

Since, in the definition of a bimodule herd, the pen $\hT$ appears only in the forms $\hT\ot_S T$ and $T\ot_R \hT$, Proposition~\ref{prop.dual} implies that {\em a posteriori}  the definition of a bimodule herd does not depend on the choice of a formal dual of $T$.

\begin{corollary}\label{cor.coring}
Let $(T,\gamma)$ be a $B$-$A$ herd.
\begin{zlist}
\item The $A$-bimodule $\cC = \hT \ot_S T\cong T^*\ot_S T$ is an $A$-coring with coproduct
$$
\DC: \hx\ot_S x\mapsto \hx\ot_S\gamma_A(x),
$$
and the counit $\eC=\ev$. $T$ is a right $\cC$-comodule with the coaction $\gamma_A$.

\item The $B$-bimodule $\cD = T \ot_R \hT\cong T\ot_R {}^*T$ is a $B$-coring with coproduct
$$
\DD: x\ot_R \hx\mapsto \gamma_B(x)\ot_R \hx,
$$
and the counit $\eD= \hev$. $T$ is a left $\cD$-comodule with the coaction $\gamma_B$.
\end{zlist}
\end{corollary}
\begin{proof} (1) The maps $\Theta_N$ in Proposition~\ref{prop.dual} establish an isomorphism of functors $\Theta: \rhom A T - \ot_S T\to - \ot_A \hT \ot_ST$. The domain of $\Theta$ is a comonad on the category of right $A$-modules, hence so is the codomain of $\Theta$. This implies that $\hT\ot_S T$ is an $A$-coring with the described comultiplication and counit.

For a less categorical proof, one can use the following direct arguments. By the definition of the map $\ev$ and Lemma~\ref{lemma.lin} both the coproduct and counit are $A$-bimodule maps. The coassociativity of $\DC$ follows immediately by diagram \eqref{eq.ass}. The equality $(\eC\ot_A\cC)\circ \DC =\cC$ is an immediate cosequence of diagram \eqref{eq.ev}. The other counitality property, $(\cC\ot_A\eC)\circ \DC =\cC$, is established by converting $\ev$ to $\hev$ with the help of the diagram \eqref{diag.B}, and then by using \eqref{eq.coev}. 

By Lemma~\ref{lemma.lin}, $\gamma_A$ is a right $A$-module map. It is coassociative by \eqref{eq.ass} and is counital by \eqref{eq.ev}.

The assertion (2) is proven by symmetric arguments.
\end{proof}

\begin{example}\label{ex.comatrix}
Take a finitely generated projective right $A$-module $T$, and set $B=\rend A T$ and $R=A$. Let  $\hT = T^*$. Then  $T$ is a bimodule herd with
$$
\gamma: T \to T\ot_A \hT \ot_S T, \quad x\mapsto \sum_i e_i\ot_A f^i\ot_S x,
$$
where $e_i\in T$, $f^i\in \hT$ is (any) finite dual basis for $T$. The coring $\cC$ is simply the (finite) comatrix coring \cite{ElKGom:com}.
\end{example}

\begin{example}\label{ex.firm}
Let $\hS$ be a ring, possibly without a unit. We say that a right $\hS$-module $M$ is {\em firm}  if and only if the multiplication map induces an isomorphism $M\ot_{\hS}\hS\to M$. A ring is called firm if it is firm as a left, or equivalently right, $\hS$-module. If $\hS$ has a unit, then firm modules are exactly the unital modules. The category of all firm right modules of $\hS$ and $\hS$-linear maps between them is denoted by $\M_{\hS}$.

A right $A$-module $T$ is said to be {\em $\hS$-firmly projective} \cite{V:equiv} if it is an $\hS$-$A$ bimodule that is firm as a left $\hS$-module, and if the functor $-\ot_{\hS}T:\M_{\hS}\to\M_A$ has a right adjoint of the form $-\ot_A\hT$, where $\hT$ is an $A$-$\hS$ bimodule that is firm as a right $\hS$-module. Denote the unit of the adjunction by $\eta$ and the counit by $\epsilon$. Then $\eta_{\hS}:\hS\to T\ot_A\hT$ and $\epsilon_A:\hT\ot_{\hS} T\to A$.
If $\hS$ has a unit, then an $\hS$-firmly projective right $A$-module is precisely a finitely generated and projective right $A$-module.

Let $T$ be an $\hS$-firmly projective right $A$-module and use notation as above. Let $S$ be the Dorroh-extension of $\hS$, which is a ring with unit. One can easily observe that $M\ot_{\hS} N\cong M\ot_S N$ for $M\in\M_{\hS}$ and $N\in{_{\hS}\M}$. Furthermore, $\hT$ is a formal dual of $T$, having $\ev=\epsilon_A$, $R=A$, $B=\Endd_A(T)$ and $\hev:T\ot_A\hT\to B$, $\hev(x\ot_A\hx)(y)=x\epsilon_A(\hx\ot_Sy)$. Finally, $T$ is a bimodule herd, where the shepherd 
$\gamma=\eta_T: T \to T\ot_S \hT\ot_AT$, is the unit of the adjunction on $T$. The associated $A$-coring $\Cc$ is the comatrix coring associated to the firm bimodule $T$ as defined in \cite{GTV:firm}. The associated $B$-coring $\Dd$ coincides with the construction of a coring out of a firm ring that is an ideal in a unital ring (see \cite[Theorem 1.6]{BohVer:firm}).
\end{example}

\begin{example}\label{ex.Galois}
Let $C$ be an $R$-coring, and let $\psi: C\ot_R A\to A\ot_R C$ be an $R$ bimodule map entwining $C$ with $A$. Set $\cC := A\ot_R C$ to be the $A$-coring associated to this entwining structure. Assume that $T$ is a finite Galois (right) comodule of $\cC$. This means that $T$ is a right $\cC$-comodule that is finitely generated and projective as a right $A$-module and that the canonical map
$$
\can : T^*\ot_S T \to \cC = A\ot_R C, \qquad f\ot_S x \mapsto f(x\sw 0)\ot_R x\sw 1,
$$
where $S = \Rend \cC T$, is bijective (an isomorphism of $A$-corings). Here $x\mapsto x\sw 0\ot_R x\sw 1$ (summation implicit) denotes the coaction of $C$ on $T$ (the $\cC$-coaction is then $x\mapsto  x\sw 0 \ot_A 1_A\ot_R x\sw 1$). Set $B = \rend AT$, $\hT = T^*$, $\hev: T\ot_R T^*\to T\ot_A T^* \cong B$, and $\ev: T^*\ot_S T\to A$ the standard evaluation. Consider the {\em translation map}
$$
\tau: C\to T^*\ot_S T, \qquad c\mapsto \can^{-1}(1_A\ot_R c).
$$
Then $T$ is a bimodule herd  with the shepherd
$$
\gamma: T\to T\ot_R T^*\ot_S T, \qquad x\mapsto x\sw 0 \ot_R \tau(x\sw 1).
$$
\end{example}
\begin{proof}
Since $\gamma$ is a composition of left  $S$-module maps, it is a left  $S$-module map. For all $a\in A$ and $c\in C$, write 
$$
\psi (c\ot_R a) = \sum_\psi a_\psi \ot_R c^\psi.
$$
The right $A$-linearity of $\can^{-1}$ implies that, for all $a\in A$ and $c\in C$, \begin{equation}\label{eq.a-lin}
\tau(c)a = \sum_\psi a_\psi \tau(c^\psi). 
\end{equation}
In particular, for all $r\in R$,
$
\tau(c)\alpha(r) = \alpha(r)\tau(c),
$
i.e.\ the image of $\tau$ is in the centraliser of $R$ in $T^*\ot_S T$. Since the coaction of $A\ot_R C$ on $T$ is right $A$-linear, and the right $A$-multiplication in $A\ot_R C$ is given through $\psi$, the equality mentioned below yields, for all $x\in T$ and $r\in R$,
\begin{eqnarray*}
\gamma(xr) &=& (xr)\sw 0 \ot_R \tau\left( \left( xr\right)\sw 0\right) = \sum_\psi x\sw 0 r_\psi \ot_R \tau\left( x\sw 1 ^\psi\right) \\
&=& x\sw 0 r\ot_R \tau(x\sw 1) = x\sw 0\ot_R \tau(x\sw 1)r = \gamma(x)r.
\end{eqnarray*}
This proves that $\gamma$ is a right $R$-module map.

Let $\{e_i\in T, \; e^*_i\in T^*\}$ be a dual basis. Identifying $B$ with $T\ot_AT^*$ we can identify $1_B$ with   $\sum_i e_i\ot_A e^*_i$. Take any $x\in T$ and apply the identity map $(T\ot_A \can^{-1})\circ (T\ot_A \can)$ to $\sum_i e_i\ot_A e^*_i\ot_S x$ to conclude that
$$
x\sw 0\ot_A \tau(x\sw 1) = \sum_i e_i\ot_A e^*_i\ot_S x.
$$
This means that the map $\gamma$ makes the diagram \eqref{eq.coev} commute. Next, take any $c\in C$, and evaluate the identity map $\can\circ\can^{-1}$ on $1_A\ot_R c$ to obtain 
$
\ev\circ\tau = \alpha\circ \eps_C.
$
This equality then yields, for all $x\in T$,
$$
x\sw 0 \ot_R \ev\left(\tau\left( x\sw 1\right)\right) = x\sw 0 \ot_R \alpha\left(\eps_C\left(x\sw 1\right)\right) = x\ot_R 1_A,
$$
i.e.\ the diagram \eqref{eq.ev} is commutative. The commutativity of diagram \eqref{eq.ass} follows by the $C$-colinearity of $\tau$.
\end{proof}

\begin{notation} Given a $B$-$A$ herd  $(T,\gamma)$ with a pen $\hT$, define an $R$-bimodule $C$ as the equaliser
\begin{equation}\label{def.C}
\xymatrix{ C\ar[r] & \hT \ot_S T\ar@<.6ex>[rrrr]^-{(\ev\ot_R\hT\ot _ST)\circ (\hT\ot_S\gamma)}
   \ar@<-.6ex>[rrrr]_-{\alpha\ot_R\hT\ot_ST}  &&&&
    A\ot_R\hT\ot_ST\, .} 
 \end{equation}
Symmetrically, define an $S$-bimodule $D$ as the equaliser
\begin{equation}\label{def.D}
\xymatrix{ D\ar[r] & T \ot_R \hT\ar@<.6ex>[rrrr]^-{(T\ot_R\hT\ot _S\hev)\circ (\gamma\ot_R\hT)}
   \ar@<-.6ex>[rrrr]_-{T\ot_R\hT\ot_S\beta}  &&&&
    T\ot_R\hT\ot_SB\, .} 
 \end{equation}
 \end{notation}
 
 \begin{proposition}\label{prop.iso}
 Let $(T,\gamma)$ be a $B$-$A$ herd. Define $C$ by the equaliser \eqref{def.C} and $D$ by the equaliser \eqref{def.D}.
 \begin{zlist}
 \item If the equaliser \eqref{def.C} is a $T_R$-pure equaliser, then
 $$
 \cC = \hT\ot_S T \cong A\ot_R C,
 $$
 as $A$-$R$ bimodules.
 
  \item If the equaliser \eqref{def.D} is a ${}_ST$-pure equaliser, then
 $$
 \cD = T\ot_R \hT \cong D\ot_S B,
 $$
 as $S$-$B$ bimodules.
 \end{zlist}
 \end{proposition}
 \begin{proof}
 (1) Set 
 $$
 \balpha = T\ot_R \alpha\ot_R\hT\ot_ST, \quad \kappa = T\ot_R \left( (\ev\ot_R\hT\ot _ST)\circ (\hT\ot_S\gamma)\right),
 $$
 and note that, for all $x\in T$,
 \begin{eqnarray*}
 \kappa\circ\gamma (x) &=& x\sut 1\ot_R \ev(x\sut 2\ot_S x\sut 3\sut 1)\ot_R x\sut 3\sut 2\ot_S x\sut 3\sut 3\\
 &=& x\sut 1\sut 1\ot_R  \ev(x\sut 1 \sut 2\ot_S x\sut 1\sut 3)\ot_R x\sut 2\ot_S x\sut 3\\
 &=& x\sut 1 \ot_R 1_A\ot_R x\sut 2\ot_S x\sut 3 = \balpha\circ\gamma(x).
 \end{eqnarray*}
 The second equality follows by \eqref{eq.ass}, and the third one is a consequence of \eqref{eq.ev}. Since the equaliser \eqref{def.C} is $T_R$-pure, and $\kappa$ and $\balpha$ are the equalised maps tensored with $T_R$, we conclude that, for all $x\in T$,
 $$
 \gamma(x) \in T\ot_R C.
 $$
 Hence we can define
 $$
 \theta: \hT\ot_S T\to A\ot_R C, \qquad \theta = (\ev \ot_R \hT\ot_S T)\circ (\hT\ot_S\gamma).
 $$
 The map $\theta$ is left $A$-linear, since $\ev$ is left $A$-linear, and it is right $R$-linear since $\gamma $ is right $R$-linear. Furthermore, the map $\theta$ is bijective with the inverse
 $$
 \theta^{-1} : A\ot_R C\to \hT\ot_S T, \qquad a \ot_R \sum_i \hx_i\ot_S x_i\mapsto \sum_i a\hx_i\ot_S x_i.
 $$
 Indeed, for all $\hx\in \hT$ and $x\in T$,
 \begin{eqnarray*}
 \theta^{-1}\circ\theta (\hx\ot_S x) &=& \ev(\hx\ot_S x\sut 1)x\sut 2\ot_S x\sut 3\\
 &=& \hx\hev(x\sut 1\ot_R x\sut 2)\ot_S x\sut 3 = \hx\ot_S x,
 \end{eqnarray*}
 where the second equality follows by \eqref{diag.B}, while the last equality is a consequence of \eqref{eq.coev}. Second, for all $a\in A$ and $\sum_i \hx_i\ot_S x_i\in C$,
 \begin{eqnarray*}
 \theta\circ\theta^{-1} (a \ot_R \sum_i \hx_i\ot_S x_i) &=& \sum_i \ev(a\hx_i\ot_S x_i\sut 1)\ot_R x_i\sut 2 \ot_S x_i\sut 3\\
 &=&
a\ev\left( \sum_i \hx_i\ot_S x_i\sut 1\right)\ot_R x_i\sut 2 \ot_S x_i\sut 3\\
& = & a\ot_R \sum_i \hx_i\ot_S x_i,
\end{eqnarray*}
where the $A$-linearity of $\ev$ is used in the first equality, and the last equality follows by the definition of $C$.

Statement (2) is proven by symmetric arguments.
\end{proof}

\begin{lemma}\label{lemma.split}
Let $(T,\gamma)$ be a $B$-$A$ herd.
\begin{zlist}
\item The equaliser \eqref{def.C} tensored with $T_R$ is a split equaliser. Consequently, if $T$ is a faithfully flat right $R$-module, then \eqref{def.C} is a pure equaliser in $\M_R$.
\item The equaliser \eqref{def.D} tensored with ${}_ST$ is a split equaliser. Consequently, if $T$ is a faithfully flat  left $S$-module, then \eqref{def.D} is a pure equaliser in ${_S\M}$.
\end{zlist}
\end{lemma}
\begin{proof}
(1) Denote the equalised maps in \eqref{def.C} by $\zeta_C$ and $\xi_C$ and set as before
 $$
 \balpha = T\ot_R \zeta_C = T\ot_R \alpha\ot_R\hT\ot_ST, \quad \kappa =T\ot_R \xi_C =  T\ot_R \left( (\ev\ot_R\hT\ot _ST)\circ (\hT\ot_S\gamma)\right).
 $$
Define
 $$
 \pi_C: T\ot_R A\ot_R \hT\ot_S T \to T\ot_R \hT\ot_S T, \qquad x\ot_R a\ot_R \hx \ot_S y\mapsto xa\ot_R \hx \ot_S y.
 $$
 Obviously, $\pi_C\circ \balpha = T\ot_R \hT\ot_S T$. Furthermore, for all $x,y\in T$ and $\hx\in \hT$,
 \begin{eqnarray*}
 \kappa\circ\pi_C\circ\kappa (y\ot_R \hx\ot_S x) &=& y\ev (\hx\ot_S x\sut 1)\ot_R \ev(x\sut 2\ot_S x\sut 3\sut 1)\ot_R x\sut 3\sut 2\ot_S x\sut 3\sut 3\\
 &=&y\ev (\hx\ot_S x\sut 1\sut 1)\ot_R \ev(x\sut 1\sut 2\ot_S x\sut 1\sut 3)\ot_R x\sut 2\ot_S x\sut 3\\
 &=& y\ev (\hx\ot_S x\sut 1)\ot_R 1_A\ot_R x\sut 2\ot_S x\sut 3\\
 &=& \balpha\circ\pi_C\circ\kappa (y\ot_R \hx\ot_S x),
 \end{eqnarray*}
 where the diagram \eqref{eq.ass} is used to derive the second equality. The third equality follows by \eqref{eq.ev}. This proves that  $T\ot_R \zeta_C$ and $T\ot_R\xi_C$ is a contractible pair, hence \eqref{def.C} tensored with $T_R$ is a split equaliser. 
 
 Assume now that $T$ is a faithfully flat right $R$-module. Since $T_R$ is flat, $T\ot_R C$ is the equaliser of $T\ot_R \zeta_C$ and $T\ot_R\xi_C$. The latter is a split, hence absolute, equaliser of right $R$-module maps, thus, for all left $R$-modules $V$,  $T\ot_R C\ot_R V$ is the equaliser of $T\ot_R \zeta_C\ot_R V$ and $T\ot_R\xi_C\ot_R V$. Since faithfully flat modules reflect equalisers, we conclude that $C\ot_R V$ is the equaliser of $ \zeta_C\ot_R V$ and $\xi_C\ot_R V$. This means that \eqref{def.C} is a pure equaliser of right $R$-module maps. 
  
 (2) This is proven by symmetric arguments. In particular, the splitting morphism is
$$
 \pi_D: T\ot_R \hT\ot_SB\ot_S T \to T\ot_R \hT\ot_S T, \qquad x\ot_R  \hx \ot_S b\ot_S y\mapsto x\ot_R \hx \ot_S by.
 $$ 
\end{proof}

Recall that, for any (unital associative) rings $K$ and $L$,  a ring map $K\to L$ is called a {\em split extension} if it is a $K$-bimodule section.

\begin{lemma}\label{lemma.split.1}
Let $(T,\gamma)$ be a $B$-$A$ herd.
\begin{zlist}
\item If $\alpha$ is a split extension, then  the equaliser \eqref{def.C}  is a split (hence pure) equaliser of $R$-bimodules. 
\item If  $\beta$ is a split extension, then  the equaliser \eqref{def.D}  is a split (hence pure) equaliser of $S$-bimodules. 
\end{zlist}
\end{lemma}
\begin{proof}
This lemma is proven by  calculations similar to that in the proof of Lemma~\ref{lemma.split}. If $\pi_\alpha: A\to R$ is an $R$-bimodule map such that $\pi_\alpha\circ\alpha = R$, then the splitting morphism for the equaliser \eqref{def.C} is 
$
\pi_C:  A\ot_R \hT\ot_S T \to  \hT\ot_S T$, $a\ot_R \hx \ot_S x\mapsto \pi_\alpha(a)\hx \ot_S x$. Symmetrically, if  $\pi_\beta: B\to S$ is an $S$-bimodule map such that $\pi_\beta\circ\beta = S$, then the splitting morphism for the equaliser \eqref{def.C} is 
$
\pi_D: T\ot_R \hT\ot_S B \to  T\ot_R\hT$, $ x \ot_R \hx\ot_S b\mapsto x \ot_R \hx\pi_\beta(b)$. 
\end{proof}

\begin{theorem}\label{thm.entw}
Let $(T,\gamma)$ be a $B$-$A$ herd, and let $C$ be defined by the equaliser \eqref{def.C} and $D$ by the equaliser \eqref{def.D}.
\begin{zlist}
\item Assume that
\begin{rlist}
\item  $T$ is a faithfully flat right $R$-module and $A$ is a faithfully flat right (or left) $R$-module, or
\item  $\alpha$ is a split extension.
\end{rlist}
Then:
\begin{blist}
\item $C$ is an $R$-coring with coproduct
$$
\Delta_C: C\to C\ot_R C, \qquad \sum_i \hx_i\ot_S x_i\mapsto \hx_i\ot_S \gamma(x_i),
$$
and counit $\eps_C = \ev\mid_ C$.
\item $C$ is entwined with $A$ (over $R$) by the map
$
\psi: C\ot_R A\to A\ot_R C$,
$$
\sum_i \hx_i\ot_S x_i\ot_R a\mapsto \sum_i \ev\left(\hx_i\ot_S (x_ia)\sut 1\right)\ot_R (x_ia)\sut 2\ot_S (x_ia)\sut 3.
$$
\item $T$ is a right $(A,C,\psi)_R$-entwined module with the coaction $\gamma$.
\end{blist}
\item Assume that
\begin{rlist}
\item $T$ is a faithfully flat left $S$-module and $B$ is a faithfully flat left (or right) $S$-module, or
\item  $\beta$ is a split extension.
\end{rlist}
Then:
\begin{blist}
\item $D$ is an $S$-coring with coproduct
$$
\Delta_D: D\to D\ot_R D, \qquad \sum_i x_i\ot_R\hx_i\mapsto  \gamma(x_i)\ot_R \hx_i,
$$
and counit $\eps_D = \hev\mid_ D$.
\item $B$ is entwined with $D$ (over $S$) by the map
$
\varphi: B\ot_S D\to D\ot_SB$,
$$
 \sum_i b\ot_Sx_i\ot_R \hx_i\mapsto \sum_i  (bx_i)\sut 1\ot_R (bx_i)\sut 2\ot_S \hev\left((bx_i)\sut 3 \ot_R \hx_i \right).
$$
\item $T$ is a left $(B,D,\varphi)_S$-entwined module with the coaction $\gamma$.
\end{blist}
\end{zlist}
\end{theorem}
\begin{proof}
(1)(a) Under either of the hypotheses,  the equaliser \eqref{def.C} is $T_R$-pure, thus, as explained in the proof of Proposition~\ref{prop.iso},  $\gamma(T)\subseteq T\ot_R C$. Consequently, $\Delta_C (C)\subseteq \hT\ot_S T\ot_R C$. Furthermore, writing as before $\zeta_C$ and $\xi_C$ for the maps equalised in \eqref{def.C},
\begin{eqnarray*}
\left(\xi_C\ot_R C\right)\circ \Delta_C &=& (\ev\ot_R \hT\ot_S T\ot_R C)\circ (\hT\ot_S\gamma\ot_R C)\circ (\hT\ot_S\gamma)\mid_C\\
&=&  (\ev\ot_R \hT\ot_S T\ot_R C)\circ (\hT\ot_ST\ot_R\hT\ot_S\gamma)\circ (\hT\ot_S\gamma)\mid_C\\
&=&  (A\ot_R \hT\ot_S\gamma)\circ (\ev\ot_R C)\circ (\hT\ot_S\gamma)\mid_C\\
&=& (A\ot_R \hT\ot_S\gamma)\circ (\alpha\ot_R C) = (\zeta_C\ot_R C)\circ \Delta_C,
\end{eqnarray*}
where the second equality follows by diagram \eqref{eq.ass}, and the fourth equality is a consequence of the definition of $C$. In view of Lemma~\ref{lemma.split} (in the case of hypothesis (i)) or Lemma~\ref{lemma.split.1} (in the case of hypothesis (ii)), the equaliser of right $R$-module maps $\zeta_C$ and $\xi_C$, i.e.\ the equaliser defining $C$, is a pure equaliser, hence $\Delta_C(C)\subseteq C\ot_R C$. Therefore, $\Delta_C$ is a well defined $R$-bimodule map  $C\to C\ot_R C$. It is coassociative by diagram \eqref{eq.ass}. 

For any $c \in C$, 
$$
(\ev\ot_R \hT\ot_S T)\circ (\hT\ot_S\gamma)(c) = 1_A\ot_R c.
$$
Applying $A\ot_R \ev$ to this equality and using \eqref{eq.ev} we immediately obtain
$$
1_A\ot_R \ev(c) = \ev(c)\ot_R 1_A.
$$
If $A$ is a faithfully flat right or left $R$-module (hypothesis (i)), the above equality implies that, for all $c\in C$, $\ev(c) \in R$. On the other hand, if there is an $R$-bimodule map $\pi_\alpha: A\to R$ such that $\pi_\alpha\circ\alpha =R$, then applying it to both sides of the above equality one concludes that $\ev(c) = (\alpha\circ \pi_\alpha\circ\ev)(c)$, i.e. $\ev(c)\in R$ as needed. That $\eps_C = \ev\mid_C$ is a counit for $\Delta_C$ follows by the definition of $C$ and  diagram \eqref{eq.ev}. 

(1)(b) and (1)(c). By either of the hypotheses, $C$ is defined by a $T_R$-pure equaliser. Thus, by Proposition~\ref{prop.iso}~(1), $A\ot_R C\cong \hT\ot_S T$ as $A$-$R$-bimodules. Using the explicit form of this isomorphism in the proof of Proposition~\ref{prop.iso}~(1), one easily finds that the induced right $A$-module structure on $A\ot_R C$ is, for all $a\in A$, $c\in C$,
$$
(1_A\ot_R c)a := \theta\left(\theta^{-1}\left(1_A\ot_R c\right) a\right) = \psi( c\ot_R a).
$$
Furthermore, the induced (i.e.\ compatible with the isomorphism $\theta$) $A$-coring structure on $A\ot_R C$ comes out as $A\ot_R \Delta_C$ and $A\ot_R\eps_C$. This implies that $C$ is entwined with $A$ by $\psi$ (cf.\ \cite[Proposition~2.1]{Brz:str}). Since $T$ is a right $\hT\ot_S T\cong A\ot_R C$-comodule with the coaction $\gamma_A$, it is a right entwined module. The induced  $C$-coaction $(T\ot_A\theta)\circ\gamma_A$ comes out as $\gamma$.

The assertions (2) are proven by symmetric arguments.
\end{proof}

\begin{remark}\label{rem.entw} 
The observations of Theorem~\ref{thm.entw} under the hypotheses (ii) are a bimodule version of the construction of a Hopf algebra from a (copointed) quantum heap in \cite{Sko:hea}. More specifically, let $H$ be a quantum heap (over a commutative ring $k$) with the structure map $\gamma: H\to H\ot_k H\ot_k H$, and let $\pi_\alpha:H\to k$ be an algebra character. Then $H$ is a $k$-$k$ bimodule herd, and let $C$ be the associated $k$-coring (coalgebra). Then the map $\pi_\alpha \ot_k H\mid_C : C\to H$ is an isomorphism of coalgebras with the inverse $(\pi_\alpha \ot_k  H\ot_k  H)\circ \gamma$.
\end{remark}

\subsubsection*{Herds and Galois comodules} The following theorem, which is the main result of this section, establishes tame $B$-$A$-herds as a way of describing finite Galois comodules.

\begin{theorem}\label{thm.torsor.galois}
Let $T$ be a  $B$-$A$ bimodule 
that is a progenerator as a right $A$-module, $B=\rend AT$, and assume that
\begin{rlist}
\item  $A$ is a faithfully flat right $R$-module and $B$ is a faithfully flat left $S$-module, or
\item  $\alpha$  and $\beta$ are split extensions.
\end{rlist}
Then the following statements are equivalent:
\begin{blist}
\item $T$ is a (tame) bimodule herd.
\item There exists a right entwining $\psi: C\ot_R A \to A\ot_R C$ over $R$ such that $T$ is a right Galois comodule over $\cC = A\ot_R C$ with $S=\Rend \cC T$.
\item There exists a left entwining $\varphi: B\ot_S D\to D\ot_S B$ over $S$ such that $T$ is a left Galois comodule over $\cD = D\ot_S B$ with $R= \Lend \cD T$.
\end{blist}
\end{theorem}
\begin{proof}
(a) $\Ra$ (b) Theorem~\ref{thm.entw} implies that there is an entwining as stated and that $T$ is a right entwined module (i.e.\  a right comodule of the coring $\cC = A\ot_R C$) with coaction $\gamma$ (note that by Corollary~\ref{cor.faithflat}, $T_R$ is faithfully flat under condition (i)). By construction, $T^*\ot_S T \cong A\ot_R C$, with the isomorphism described in the proof of Proposition~\ref{prop.iso} which, with the choice of the coaction on $T$, coincides with the map $\can$. Thus it only remains to identify $S$ with the endomorphism ring $\Rend \cC T$. Since $B = \rend A T$, $\Rend \cC T$ is a subalgebra of $B$ consisting of all $s\in B$ such that, for all $x\in T$,
$$
s\gamma(x) = \gamma (sx).
$$
Obviously, $S\subseteq \Rend \cC T$. Apply the map $\hev\ot_S T$ to this equality and use diagram \eqref{eq.coev} to find that 
\begin{equation}\label{SisEndCT}
s\ot_S x =  1_B \ot_Ssx.
\end{equation} 

If hypothesis (i) holds, then $B$ is faithfully flat as a left $S$-module  and -- by the fact that $B$ is an endomorphism ring of a progenerator -- $T$ is a progenerator of left $B$-modules, $T$ is also faithfully flat  as a left $S$-module. Thus  the equality $s\ot_S x =  1_B \ot_Ssx$, for all $x\in T$, implies that $s\in S$, hence $S=\Rend \cC T$.

On the other hand, suppose that there is an $S$-bimodule map $\pi_\beta: B\to S$, such that $\pi_\beta\circ \beta= S$. Combining $\pi_\beta$ with the inclusion $\Rend \Cc T\subseteq \rend A T=B$, we obtain a map $\pi:\Rend \Cc T\to S$. Clearly, $\pi$ is a retraction for the inclusion $S\subseteq \Rend \cC T$. If we apply $\pi_\beta\ot_ST$ to \eqref{SisEndCT}, then we find $\pi(s)x= sx$, which means exactly that $S= \Rend \cC T$.

(b) $\Ra$ (a) Follows by Example~\ref{ex.Galois}. 

The equivalence of (a) and (c) is proven by symmetric arguments.
\end{proof}
 
 \begin{remark}\label{rem.torsor.galois}
 With the assumptions of Theorem~\ref{thm.torsor.galois}, there is a bijective correspondence between the following sets:
 \begin{blist}
 \item the set of shepherds $\gamma: T\to T\ot_R T^*\ot_S T$;
 \item the set of right entwining structures $\psi: C\ot_R A \to A\ot_R C$ over $R$ such that $T$ is a right Galois comodule over $\cC = A\ot_R C$ with $S=\Rend \cC T$;
\item the set of left entwining structures $\varphi: B\ot_S D\to D\ot_S B$ over $S$ such that $T$ is a left Galois comodule over $\cD = D\ot_S B$ with $R= \Lend \cD T$.
\end{blist}

Starting with $\gamma$ one constructs the $R$-coring $C\subseteq T^*\ot_S T$ and entwining $\psi$ as in Theorem~\ref{thm.entw}. The translation map $\tau : C\to T^*\ot_S T$ (cf.\ Example~\ref{ex.Galois}) is simply the obvious inclusion, and since the $C$-coaction on $T$ is given by $\gamma$, the procedure of obtaining a shepherd from $\tau$ described in Example~\ref{ex.Galois} reproduces $\gamma$.

Starting with an entwining map $\psi$ and the translation map $\tau: C\to T^*\ot_S T$, one defines $\gamma$ as in Example~\ref{ex.Galois}. Using the fact that $\tau(c) = \can^{-1}(1_A\ot_R c)$ one easily finds that the image of $\tau$ is in the $R$-coring $\bar{C}$ defined by equaliser \eqref{def.C}. The corestriction of $\tau$ establishes then an isomorphism of $C$ with $\bar{C}$. Explicitly, the inverse of $\tau$ is $\bar{C}\ni \sum_i f^i\ot_S x^i\mapsto \sum_i f^i(x^i\sw 0)x^i\sw 1$. By Theorem~\ref{thm.entw} there is an entwining map $\bar{\psi} : \bar{C}\ot_R A\to \bar{C}\ot_R A$. Using the $A$-linearity of $\tau$, \eqref{eq.a-lin}, one finds that the composition
$$
\xymatrix{ C\ot_R A \ar[rr]^{\tau\ot_R A} && \bar{C}\ot_R A \ar[rr]^{\bar{\psi}} && A\ot_R\bar{C} \ar[rr]^{A\ot_R \tau^{-1}} && A\ot_R C,}
$$
equals $\psi$. 
\end{remark}

\begin{remark}
In some interesting situations, condition (ii) of Theorem \ref{thm.torsor.galois} implies already condition (i). This can be seen as follows. Let $T$ be a $B$-$A$ bimodule that is a progenerator as right $A$-module. Then, by applying the Hom-tensor relations,  we obtain the following natural isomorphisms
$${_S\Hom}(T,-)\simeq  {_S\Hom}(B\ot_BT,-) \simeq {_S\Hom}(B,{\Hom_A}(T,-))$$
and 
$${_S\Hom}(B,-)\simeq {_S\Hom}(T\ot_A \hT,-)\simeq {_S\Hom}(T,\Hom_B(\hT,-))\ .$$
Therefore, ${_SB}$ is projective if and only if ${_ST}$ is projective. 
Under this projectivity condition, ${_SB}$ is faithfully flat if and only if $\beta$ is a split monomorphism of left $S$-modules; see \cite[2.11.29]{Row:rin}. 

Similarly one proves that ${A_R}$ is projective if and only if $T_R$ is projective. Under this condition, $A_R$ is faithfully flat if and only if $\alpha$ is a split monomorphism of right $R$-modules.

In particular, if ${_SB}$ and ${A_R}$ are projective and, $\alpha$ and $\beta$ are split extensions (condition (ii) of Theorem \ref{thm.torsor.galois}), then ${_SB}$ and ${A_R}$ are faithfully flat (condition (i)).
\end{remark}

\section{Herds versus coherds}\label{se:co-torsors}
\setcounter{equation}0
By formally dualising the definition of bimodule herds, the notion of a {\em bicomodule coherd} is introduced. It is shown that a tame bimodule herd is also a bicomodule coherd of corresponding corings.

\subsubsection*{Bicomodule coherds}

\begin{definition}
Let $C$ be an $R$-coring and $D$ an $S$-coring. Consider a bicomodule $(X,\rho^{D,X},\rho^{X,C})\in{^D\M^C}$. A bicomodule $(\oX,\rho^{C,\oX},\rho^{\oX,D})\in{^C\M^D}$ is called a \emph{companion} of $X$ if there exist a $C$-bicomodule map
$$\cov:C\to \oX\ot_SX,$$
and a $D$-bicomodule map
$$\ocov:D\to X\ot_R\oX, $$
such that the following diagrams commute,
\begin{equation}\label{diag.X}
\xymatrix{
X \ar[rr]^-{\rho^{X,C}} \ar[d]_{\rho^{D,X}} && X\ot_RC \ar[d]^{X\ot_R\cov}\\
D\ot_SX \ar[rr]_-{\ocov\ot_SX} && X\ot_R\oX\ot_SX \ ,
}
\end{equation}
\begin{equation}\label{diag.oX}
\xymatrix{
\oX \ar[rr]^-{\rho^{C,\oX}} \ar[d]_{\rho^{\oX,D}} && C\ot_R\oX \ar[d]^{\cov\ot_R\oX}\\
\oX\ot_SD \ar[rr]_-{\oX\ot_S\ocov} && \oX\ot_SX\ot_R\oX \ .
}
\end{equation}
Furthermore,  a $D$-$C$ bicomodule $X$ with a companion $\oX$ is called a {\em bicomodule coherd} if there exists an $S$-$R$ bimodule map 
$$\chi:X\ot_R\oX\ot_SX\to X,$$
rendering commutative the following diagrams,
\begin{equation}\label{diag.C}
\xymatrix{
X\ot_RC \ar[d]_{X\ot_R\varepsilon_C} \ar[rr]^-{X\ot_R\cov} && X\ot_R\oX\ot_SX \ar[d]^{\chi} \\
X\ot_RR \ar[rr]_-{\cong} && X \ ,
}
\end{equation}
\begin{equation}\label{diag.D}
\xymatrix{
D\ot_SX \ar[d]_{\varepsilon_D\ot_SX} \ar[rr]^-{\ocov\ot_SX} && X\ot_R\oX\ot_SX \ar[d]^{\chi} \\
S\ot_SX \ar[rr]_-{\cong} && X \ ,
}
\end{equation}
\begin{equation}\label{diag.coass}
\xymatrix{
X\ot_R\oX\ot_SX\ot_R\oX\ot_SX \ar[d]_{\chi\ot_R\oX\ot_SX} \ar[rr]^-{X\ot_R\oX\ot_S\chi}  && X\ot_R\oX\ot_SX \ar[d]^{\chi} \\
X\ot_R\oX\ot_SX \ar[rr]_-{\chi} && X \ .
}
\end{equation}
\end{definition}

The theory of herds as developed in Section \ref{se.torsors} can now be formally dualised. In particular, given  a bicomodule coherd $X$, the $C$-bicomodule  $\oX\ot_SX$  is a non-unital ring (over $R$) with  multiplication 
$$
\mu_X:= \oX \ot_S \chi : \oX\ot_S X \ot_R \oX \ot_S X \to \oX \ot_S X .
$$
The map $\cov : C \to \oX\ot_SX$  is a {\em $C$-unit} for $\oX\ot_SX$, i.e.\ the following diagram
is commutative
$$
\xymatrix{ C\ot_R\oX\ot_SX \ar[d]_{\cov \ot_R \oX\ot_SX} && \oX\ot_SX \ar[d]^{\oX\ot_SX} \ar[ll]_-{\rho^{C,\oX}\ot_S X} \ar[rr]^-{\oX\ot_S\rho^{X,C}}
&& \oX\ot_SX\ot_R C \ar[d]^{\oX\ot_SX\ot_R\cov} \\
\oX\ot_SX\ot_R \oX\ot_SX \ar[rr]_-{\mu_X} && \oX\ot_SX && \oX\ot_SX\ot_R \oX\ot_SX \ar[ll]^-{\mu_X} \ .}
$$

Symmetrically, $X\ot_R\oX$ is a ring with product $\chi \ot_R \oX$ and with
a $D$-unit $\ocov$.
Furthermore, one can define an $R$-bimodule $A'$ as the following coequaliser
\[
\xymatrix{
C\ot_R\oX\ot_SX \ar@<.5ex>[rrrr]^-{(\oX\ot_S\chi)\circ(\cov\ot_R\oX\ot_SX)} \ar@<-.5ex>[rrrr]_-{\varepsilon_C\ot_R\oX\ot_SX} &&&& \oX\ot_SX \ar[rr]^{\pi_A} && A' \ .
}
\]
Since the tensor functor preserves coequalisers, the map $\mu_X$ descents to 
the associative product $\mu_{A'} : A'\ot_R A'\to A'$, by the formula
$$
\mu_{A'} \circ (\pi_A \ot_R \pi_A) = \pi_A \circ \mu_X.
$$
Suppose $C$ is faithfully flat as a left $R$-module, then by a (dual) descent argument, we can construct a unit for the $R$-ring $A'$ as follows.
Consider the following split coequaliser of $R$-bimodules
$$
\xymatrix{
C\ot_RC \ot_RC \ar@<2ex>[rrrr]^-{C\ot_R\eps_C\ot_R C} \ar@<-.5ex>[rrrr]^-{\eps_C\ot_R C\ot_R C} &&&& C\ot_R C\ar@<1.5ex>[llll]^-{C\ot_R\Delta_C} \ar@<.5ex>[rr]^-{\eps_C\ot_R C} &&C\ar@<.5ex>[ll]^-{\Delta_C} \ .
}
$$
Since tensoring with a faithfully flat module reflects coequalisers, one obtains
the following coequaliser or $R$-bimodules
$$
\xymatrix{
C\ot_RC  \ar@<.5ex>[rrrr]^-{C\ot_R\eps_C} \ar@<-.5ex>[rrrr]_-{\eps_C\ot_R C} &&&& C \ar[rr]^-{\eps_C} &&R \ .
}
$$
By the universal property of coequalisers there exists a unique $R$-bimodule map
$\alpha': R\to A'$ such that
$$
\pi_A\circ\cov = \alpha'\circ\eps_C.
$$
One easily checks that, for all $a\in A'$, $\mu_{A'}(a\ot_R \alpha'(1_R)) = \mu_{A'}(\alpha'(1_R)\ot_R a) = a$, i.e.\ that $A'$ is a (unital) $R$-ring with the unit map $\alpha'$.

In a symmetric way, if $D$ is a faithfully flat left or right $S$-module  one obtains the (unital) $S$-ring $B'$ as the coequaliser
\[
\xymatrix{
X\ot_R\oX\ot_SD \ar@<.5ex>[rrrr]^-{(\chi\ot_R\oX)\circ(X\ot_R\oX\ot_S\ocov)} \ar@<-.5ex>[rrrr]_-{X\ot_R\oX\ot_S\varepsilon_D} &&&& X\ot_R\oX \ar[rr]^{\pi_B} && B' \ .
}
\]
The unit map in $B'$ is the unique morphism $\beta': S\to B'$ such that $\pi_B\circ\ocov = \beta'\circ\varepsilon_D$.

\subsubsection*{Construction of coherds}

Given an $A$-coring $\cC$ and a right $\cC$-comodule $T$, set $S = \Rend \cC T$. By the {\em strong structure theorem} for $T$ is meant that the functor $-\ot_ST$ is an equivalence of the categories $\M_S$ and $\M^\cC$. 

\begin{lemma}\label{lemma.coinvariants}
Let $\Cc$ be an $A$-coring, $T$ a right $\Cc$-comodule for which the strong structure theorem holds. Then for all $N\in\M_A$ and $M\in{_A\M^\Cc}$,
the canonical morphism
$$N\ot_A\Hom^\Cc(T,M)\to \Hom^\Cc(T,N\ot_AM)$$
is an isomorphism.
\end{lemma}

\begin{proof}
This follows by a double application of the equivalence of categories between $\M_S$ and $\M^\Cc$ through the functors $-\ot_ST$ and $\Hom^\Cc(T,-)$,
$$N\ot_A\Hom^\Cc(T,M)\ot_ST\cong N\ot_AM\cong \Hom^\Cc(T,N\ot_AM)\ot_ST.$$
Since $T$ is faithfully flat as a left $S$-module, the claim follows immediately.
\end{proof}

Let $(T,\gamma)$ be a tame $B$-$A$ herd.
Then we can consider the $R$-coring $C$, which is entwined with the $R$-ring $A$ by $\psi$ and the $S$-coring $D$ which is entwined with the $S$-ring $B$ by $\phi$ as in Theorem \ref{thm.entw}. Denote as before $\Cc=A\ot_RC$ and $\Dd=D\ot_SB$ for the associated $A$-coring and $B$-coring. Recall from \cite[32.8 (2)]{BrzWis:cor} that $C\ot_RA$ is a right $\Cc$-module (i.e.\ a right entwined module): the right $A$-module structure is given by $C\ot_R\mu_A$, where $\mu_A$ is the multiplication on $A$, and the right $C$-coaction is given by $(C\ot_R\psi)\circ(\Delta_C\ot_RA)$. 
For an element $\hx\ot_Sx\ot_Ra\in C\ot_RA$ (representing a finite sum of simple tensors), the right $C$-coaction reads explicitly as 
\begin{equation}\label{eq.CAcoaction}
\varrho^{C\ot_RA}(\hx\ot_Sx\ot_Ra)=\hx\ot_Sx\sut 1\ot_R\ev(x\sut 2\ot_S (x\sut 3a)\sut 1 )\ot_R(x\sut 3a)\sut 2\ot_S (x\sut 3a)\sut 3.
\end{equation}
Symmetrically, $B\ot_SD$ is a left $\Dd$-comodule.

\begin{theorem}\label{thm.barT}
Let $(T,\gamma)$ be a tame $B$-$A$ herd. Consider the $R$-$S$ bimodule $\oT=(C\ot_R\hT)\cap (\hT\ot_S D)$. Then 
\begin{zlist}
\item $h_1: \oT\to \Hom^\Cc(T,C\ot_RA),\ 
h_1( \bar{x})( y) = (\hT\ot_S T\ot_R\ev)(\bar{x} \ot_S y)$, 
is an isomorphism of $R$-$S$ bimodules;
\item $h_2: \oT\to {^\Dd\Hom}(T,B\ot_SD),\ 
h_2( \bar{x})( y) = (\hev\ot_S T\ot_R\hT)(y\ot_R\bar{x} )$
 is an isomorphism of $R$-$S$ bimodules;
\item $h:\oT\to \hT,\ 
h=\ev\ot_A \hT\mid_{\oT} = \hT \ot_B \hev\mid_{\oT}$,
 is an $R$-$S$ bimodule map; 
\item $\oT$ is a $C$-$D$ bicomodule.
\end{zlist}
\end{theorem}

\begin{proof}
(1) Elements of  $\Hom^\Cc(T,C\ot_RA)$ are exactly right $A$-linear and right $C$-colinear morphisms $T\to C\ot_R A$. Since $\ev$ is right $A$-linear,  for any $\bar{x}\in \oT$, $h_1(\bar{x})$ is right $A$-linear as well. To check that $h_1(\bar{x})$ is right $C$-colinear, write
$\bar{x} = \hx\ot_Sx\ot_R\hy\in \oT$ (summation implicit), 
and calculate,
\begin{eqnarray*}
&&\hspace{-0.5cm}\rho^{C\ot_RA}(h_1(\hx\ot_Sx\ot_R\hy)(y))\\
&&\hspace{-0.5cm}=\hx\!\ot_S\!x\sut 1\!\ot_R\!\ev(x\sut 2\!\ot_S\! (x\sut 3\ev(\hy\!\ot_S\!y))\sut 1 )\!\ot_R\!(x\sut 3\ev(\hy\!\ot_S\!y))\sut 2\!\ot_S\! (x\sut 3\ev(\hy\!\ot_S\!y))\sut 3\\
&&\hspace{-0.5cm}=\hx\!\ot_S\!x\sut 1\!\ot_R\!\ev(x\sut 2\!\ot_S\! (\hev(x\sut 3\!\ot_R\!\hy)y)\sut 1 )\!\ot_R\!(\hev(x\sut 3\!\ot_R\!\hy)y)\sut 2\!\ot_S\! (\hev(x\sut 3\!\ot_R\!\hy)y)\sut 3\\
&&\hspace{-0.5cm}=\hx\!\ot_S\!x\!\ot_R\!\ev(\hy\!\ot_S\! y\sut 1 )\!\ot_R\!y\sut 2\!\ot_S\! y\sut 3
=h_1(\hx\!\ot_S\!x\!\ot_R\!\hy)(y\sut 1)\!\ot_R\!y\sut 2\!\ot_S\! y\sut 3\\
&&\hspace{-0.5cm}=(h_1(\hx\!\ot_S\!x\!\ot_R\!\hy)\!\ot_R\!\hT\!\ot_S\!T)(\gamma(y)),
\end{eqnarray*}
where we used \eqref{eq.CAcoaction} in the first equation,  diagram \eqref{diag.A} in the second equality and the defining property of $D$ applied on the element $\hx\ot_Sx\ot_R\hy\in\oT\subset \hT\ot_SD$ in the third equality.
Therefore, $h_1(\bar{x})$ is right $C$-colinear. Since $\hev$ is surjective, there are
$e_i\in T$ and $\he_i\in \hT\cong T^*$, such that $\hev (\sum_i e_i\ot_R \he_i) = 1_B$.
Hence it is possible to define a map
$$
k_1:\Hom^\Cc(T,C\ot_RA)\to \oT,\qquad  \varphi\mapsto \left (\varphi\ot_A\hT\right)\left(\sum_i e_i\ot_A\he_i\right).
$$
Diagram \eqref{diag.B} and the property $\hev(\sum_i e_i\ot_R\he_i) =1_B$ immediately 
imply that $k_1\circ h_1 = \oT$. 
In the other direction,
$$
h_1\circ k_1(\varphi)(x)=\sum_i\varphi(e_i)\ev(\he_i\ot_S x)=\sum_i \varphi(e_i\ev(\he_i\ot_S x))=\varphi(x), 
$$
by the right $A$-linearity of $\varphi$, diagram \eqref{diag.A} and  $\hev(\sum_i e_i\ot_R\he_i) =1_B$.

(2) This is proven by symmetric arguments.

(3) Obvious.

(4) We first prove that $\oT$ is a left $C$-comodule. By Theorem \ref{thm.entw}, 
$C$ is an $R$-coring with comultiplication $T\ot_R\gamma\mid_C$, hence $(T\ot_R\gamma)(C)\subset C\ot_RC$. Similarly, $D$ is an $S$-coring with comultiplication $\gamma\ot_RT\mid_D$ hence $(\gamma\ot_ST)(D)\subset D\ot_SD$. Therefore, it follows that $(T\ot_R\gamma\ot_ST)(\oT)\subset (C\ot_RC\ot\hT)\cap (\hT\ot_SD\ot_SD)$. Consider $C\ot_RC\ot_RA$ as a right $\Cc$-comodule with coaction $C\ot_R \varrho^{C\ot_RA}$. By a similar computation as for $h_1$, we find that the map 
\begin{eqnarray*}
&\bar{h}:\!\!\! &(T\!\ot_R\!\gamma\!\ot_S\!T)(\oT)\subset (C\!\ot_R\! C\! \ot_R\! \hT)\cap (\hT\! \ot_S\! D\! \ot_S\! D)\to \Hom^\Cc(T,C\!\ot_R\! C\!\ot_R\! A),\\
&&\bar{h}(c\ot_R c'\ot_R \hx)(x)=c\ot_R c'\ot_R \ev(\hx\ot_S x) ,
\end{eqnarray*}
is well-defined. 
Applying Lemma \ref{lemma.coinvariants}, we therefore find a well-defined map 
$$
\bar{h}\circ (T\ot_R\gamma\ot_RT):\oT\to \Hom^\Cc(T,C\ot_RC\ot_RA)\cong C\ot_R\Hom^\Cc(T,C\ot_RA).
$$
Hence $(C\ot_R k_1)\circ \bar{h}\circ (T\ot_R\gamma\ot_RT):\oT\to C\ot_R\oT$ defines a comultiplication on $\oT$. Up to an isomorphism this is just the restriction of the map $T\ot_R\gamma\ot_RT$. Coassociativity and counitiality now follow immediatelly from the diagrams \eqref{eq.ass} and \eqref{eq.ev}. By symmetric arguments one shows that $\oT$ is a right $D$-comodule. The coassociativity between left $C$- and right $D$-coaction follows from diagram \eqref{eq.ass}.
\end{proof}

\begin{theorem}\label{thm.cotor}
Let $(T,\gamma)$ be a tame $B$-$A$ herd.
Consider  corings $C$ and $D$ of Theorem \ref{thm.entw}. Then the 
$R$-$S$ bimodule $\oT=(C\ot_R\hT)\cap (\hT\ot_S D)$ of Theorem \ref{thm.barT} is a companion of $T$ and 
$(T,\chi)$ is a $D$-$C$ coherd, where $\chi:T\ot_R\oT\ot_ST\to T$ is given by
\begin{equation}
\label{chi}
\chi(x\ot_R\hx\ot_Sy\ot_R\hy\ot_Rz)=\hev(x\ot_R\hx)\, y \, \ev(\hy\ot_Rz),
\end{equation}
for all  $x, z\in T$ and $\hx\ot_Sy\ot_R\hy\in \oT\subset \hT\ot_ST\ot_R\hT$ (summation implicit).
\end{theorem}

\begin{proof}
We  first prove that $\oT$ is a companion for $T$. To this end, we must define a $C$-bicomodule map $\cov:C\to \oT\ot_S T$. By means of the canonical inclusion $\iota:C\to \hT\ot_ST$, from the definition of $C$ as an equaliser, we can consider $\Delta_C$ as a map $\Delta_C:C\to C\ot_R\hT\ot_ST$. Furthermore, as $\gamma(T)\subset D\ot_ST$, we know that $\Delta_C(C)=(\hT\ot_S\gamma)(C)\subset \hT\ot_SD\ot_ST$. Since $T$ is flat as a left $S$-module the functor $-\ot_ST$ preserves all limits, so in particular intersections. Therefore 
$$(C\ot_R\hT\ot_ST)\cap (\hT\ot_SD\ot_ST) = ((C\ot_R\hT)\cap(\hT\ot_SD))\ot_ST=\oT\ot_ST.
$$
 This defines a $C$-bicolinear map 
$$
\cov=\Delta_C:C\to \oT\ot_ST.
$$
By similar arguments, 
$$
\ocov=\Delta_D:D\to T\ot_R\oT 
$$
is well-defined and is clearly $D$-bicolinear.
Diagram \eqref{diag.X} is now exactly diagram \eqref{eq.ass}. Diagram \eqref{diag.oX} follows from \eqref{eq.ass} tensored on the left with $\hT\ot_R-$ and on the right with $-\ot_S\hT$.

Now take any $x\ot_R\hx\ot_Sy\in T\ot_RC$ (summation implicit). The condition of diagram \eqref{diag.C}  comes out as
$$
x\,\ev(\hx\ot_Sy)=\hev(x\ot_R \hx)\,y\sut 1\ev(y\sut 2\ot_S y\sut 3).
$$
Similarly, for all $x\ot_R\hx\ot_Sy\in D\ot_RT$ (summation implicit), diagram \eqref{diag.D} commutes since
$$
\hev(x\ot_S\hx)\,y=\hev(x\sut 1\ot_R x\sut 2)x\sut 3\ev(\hx\ot_S y).
$$
Finally, diagram \eqref{diag.coass} commutes because of the bilinearity of $\ev$ and $\hev$ and by diagrams \eqref{diag.A} and \eqref{diag.B}.
\end{proof}

\subsubsection*{Reconstruction of the herd}
Let $(T,\gamma)$ be a tame $A$-$B$ herd. By Theorem~\ref{thm.cotor}, $\ol{T}=(C\ot_R\hT)\cap(\hT\ot_SD)$ is a companion of $T$ and $(T,\chi)$ is a $D$-$C$ coherd, where $C$ and $D$ are corings of Theorem~\ref{thm.entw}, and $\chi:T\ot_R\ol{T}\ot_ST\to T$ is given by \eqref{chi}. Furthermore, we know from the first part of this section that we can construct the $R$-bimodule $A'$ as the following coequaliser
\[
\xymatrix{
C\ot_R\oT\ot_ST \ar@<.5ex>[rrrr]^-{(\oT\ot_S\chi)\circ(\cov\ot_R\oT\ot_ST)} \ar@<-.5ex>[rrrr]_-{\varepsilon_C\ot_R\oT\ot_ST} &&&& \oT\ot_ST \ar[rr]^{\pi_{A}} && A' \ .
}
\]
Recall that $A'$ is in general a non-unital $R$-ring, but if $C$ is faithfully flat as a left or right $R$-module (i.e., if $\hT$ is faithfully flat as a left $R$- or right $S$-module), then $A'$ has a unit. 

Put $\omega=(\oT\ot_S\chi)\circ(\cov\ot_R\oT\ot_ST)-\varepsilon_C\ot_R\oT\ot_ST$. Then $A'=\oT\ot_ST/\im\omega$ consists of classes  that satisfy
\begin{equation}\label{eq:omega}
[\ev(\hx\ot_Sx)\hy\ot_Sy\ot_R\hz\ot_Sz]
=[\hx\ot_Sx\ot_R\hy\ot_Sy\ev(\hz\ot_Sz)],
\end{equation}
(summation implicit). This follows by the defining property of $D$.

Similarly, there is a(non-unital) $S$-ring $B'$ given by the coequaliser
\[
\xymatrix{
T\ot_R\oT\ot_SD \ar@<.5ex>[rrrr]^-{(\chi\ot_R\oT)\circ(T\ot_R\oT\ot_S\ocov)} \ar@<-.5ex>[rrrr]_-{T\ot_R\oT\ot_S\varepsilon_D} &&&& T\ot_R\oT \ar[rr]^{\pi_B} && B' \ .
}
\]

\begin{theorem}\label{reconstr}
Let $(T,\gamma)$ be a tame $A$-$B$ herd and let $A'$, $B'$ be rings constructed above.
\begin{enumerate}
\item There are ring morphisms $\nu_A:A'\to A$ and $\nu_B:B'\to B$.
\item If the map $h:\oT\to \hT$ of Theorem \ref{thm.barT} (3) is an isomorphism, then $\nu_A$ and $\nu_B$ are isomorphisms, in particular $A'$ and $B'$ are unital rings.
\item If $\hT$ is flat as left $R$- and right $S$-module and
$\oT=\hT\ot_SD=C\ot_R\hT$, then maps $\nu_A$ and $\nu_B$ are isomorphisms of rings.
\end{enumerate}
\end{theorem}
\begin{proof}
We only prove the satements for  rings $A$ and $A'$. The statements for $B$ and $B'$ are verified by  symmetric arguments.

(1) Consider the map $\ol{\ev}:\oT\ot_ST\to A$, given by
$$\ol{\ev}(\hx\ot_Sx\ot_R\hy\ot_Sy)=\ev(\hx\ot_Sx)\ev(\hy\ot_Sy),$$
then obviously, $\ol{\ev} \circ\omega = 0$. Hence by the universal property of coequalisers, there is a map $\nu_A:A'\to A$. Using the properties of the evaluation maps, it is easily checked that $\nu_A$ is a ring morphism.

(2) Consider the following diagram with coequalisers as rows:
\[
\xymatrix{
C\ot_R\oT\ot_ST \ar[d]^{Q\ot_Rh\ot_ST} \ar@<.5ex>[rrrr]^-{(\oT\ot_S\chi)\circ(\cov\ot_R\oT\ot_ST)} \ar@<-.5ex>[rrrr]_-{\varepsilon_C\ot_R\oT\ot_ST} &&&& \oT\ot_ST \ar[d]^{h\ot_ST} \ar[rr]^{\pi_{A}} && A' \ar[d]^{\nu_Q} \\
C\ot_R\hT\ot_ST \ar@<.5ex>[rrrr]^-{(\hT\ot_S\mu_{T,A})\circ (C\ot_R\ev)} \ar@<-.5ex>[rrrr]_-{\varepsilon_C\ot_R\hT\ot_ST} &&&& \hT\ot_ST \ar[rr]^{\pi_{Q}} && Q \ .
}
\]
Here $\mu_{T,A}:T\ot_RA\to T$ denotes the action of $A$ on $T$.
One can check that the diagram is commutative, hence the map $\nu_Q$ exists by the universal property of  coequalisers. Since $h$ is an isomorphism, $\nu_Q$ is an isomorphism as well. We claim that $A\cong Q$. Put $\varpi=(\hT\ot_S\mu_{T,A})\circ (C\ot_R\ev)-\varepsilon_C\ot_R\hT\ot_ST$. Then $Q= \hT\ot_ST/\im\varpi$ consists of classes of elements that satisfy
\begin{equation}\label{eq:Q}
[\ev(\hx\ot_S x)\hy\ot_S y]=[\hx\ot_Sx\ev(\hy\ot_Sy)].
\end{equation}
Obviously $\ev:\hT\ot_ST\to A$ satisfies $\ev \circ \varpi=0$, hence the universal property of coequalisers yields  a map $\nu:Q\to A$. Conversely, define a map $A\to Q$ as follows. By assumption, the map $\ev$ is surjective, therefore, for all $a\in A$, there exists a (not necessarily unique) element $\hx_a\ot_S x_a\in \hT\ot_S T$ (summation implicit) such that $a=\ev(\hx_a\ot_S x_a)$. Define 
$\nu':A\to Q,\ \nu'(a)=[\hx_a\ot_S x_a]$. 
Take another element $\hy_a\ot_S y_a\in\hT\ot_S T$ such that $\ev(\hy_a\ot_S y_a)=a$, and use the defining property of $Q$ \eqref{eq:Q} to compute
\begin{eqnarray*}
[\hx_a\ot_S x_a]&=&[\hx_a\ot_S x_a\cdot 1]=[\hx_a\ot_S x_a\ev(\hx_1\ot_S x_1)]\\
&=&[\ev(\hx_a\ot_S x_a)\hx_1\ot_S x_1]=[\ev(\hy_a\ot_S y_a)\hx_1\ot_S x_1]
=[\hy_a\ot_S y_a].
\end{eqnarray*}
Thus the map $\nu'$ is well-defined. 
Obviously  $\nu\circ \nu'= A$, and a similar computation to the one above shows that $\nu'\circ \nu=Q$. Hence $A\cong Q\cong A'$.

(3) Under these conditions, the defining coequaliser diagram for $A'$ reduces to
\[
\xymatrix{
C\ot_R\hT\ot_SD\ot_ST \ar@<.5ex>[rrrr]^-{(\oT\ot_S\chi)\circ(\cov\ot_R\hT\ot_SD\ot_ST)} \ar@<-.5ex>[rrrr]_-{\varepsilon_C\ot_R\hT\ot_SD\ot_ST} &&&& \hT\ot_SD\ot_ST \ar[rr]^-{\pi_{A}} && A' \ .
}
\]
Since $\gamma(T)\subset D\ot_ST$,  the flatness of $\hT$ as as right $S$-module implies that $\hT\ot_S\gamma(\hT\ot_ST)\subset \hT\ot_SD\ot_ST$.
Using a similar notation as in the proof of part (2), define a map $\nu'_A:A\to A'$ by
$\nu'_A(a)=[\hx_a\ot_S\gamma(x_a)]$. This map is well-defined since, 
\begin{eqnarray*}
[\hx_a\ot_Sx_a\sut 1\ot_Rx_a\sut 2\ot_Sx_a\sut 3]&=&[\hx_a\ot_Sx_a\sut 1\ot_Rx_a\sut 2\ot_Sx_a\sut 3\ev(\hx_1\ot_S x_1)]\\
&&\hspace{-3cm}=[\ev(\hx_a\ot_Sx_a\sut 1)x_a\sut 2\ot_Sx_a\sut 3\ot_R\hx_1\ot_S x_1\sut 1\ev(x_1\sut 2\ot_Sx_1\sut 3)]\\
&&\hspace{-3cm}=[\ev(\hx_a\ot_Sx_a\sut 1)\ev(x_a\sut 2\ot_Sx_a\sut 3)\hx_1\ot_S x_1\sut 1\ot_Rx_1\sut 2\ot_Sx_1\sut 3]\\
&& \hspace{-3cm}=[\ev(\hy_a\ot_Sy_a
)\hx_1\ot_S x_1\sut 1\ot_Rx_1\sut 2\ot_Sx_1\sut 3]=[\hy_a\ot_S\gamma(y_a)] ,
\end{eqnarray*}
where $\hy_a\ot_Sy_a\in \hT\ot_ST$ is any element such that $\ev(\hy_a\ot_Sy_a)=\ev(\hx_a\ot_Sx_a)=a$. It is easily checked that $\nu'_A$ is the  inverse of $\nu_A$.
\end{proof}

\begin{remark}\label{rem.reconstr}
Theorem~\ref{reconstr} shows that there are (at least) two situations in which the original base rings $A$ and $B$ of the tame bimodule herd $T$ can be reconstructed from the associated coherd. Both cases have  non-empty sets of examples. 

The situation of Theorem \ref{reconstr} (2) occurs when $T=A=B$ is a ring and the associated entwining maps $\psi$ and $\phi$ of Theorem \ref{thm.entw} are bijective. This is the case described in \cite[Theorem 4.9]{BohBrz:pre}.

An explicit example of the situation of Theorem~\ref{reconstr}~(3) will be discussed in Section~\ref{se:co-object}, where we consider Galois co-objects. As in this situation $R=S=k$ is a commutative ring, and $T\cong \hT$ as $k$-module, the flatness conditions on $\hT$ are already contained in the flatness of $T$.
\end{remark}

\section{Composition of Herds}\label{se:composition}
\setcounter{equation}0
The aim of this section is to describe a way in which two bimodule herds can be composed by means of the tensor product. 

Consider in addition to the ring morphisms $\alpha:R\to A$ and $\beta:S\to B$ a third ring morphism $\kappa:Z\to K$. 

\begin{lemma}\label{lemma.composition.dual}
Let $T$ be a $B$-$A$ bimodule with a formal dual $\hT$, and let $P$ be an $A$-$K$ bimodule with a formal dual $\hP$. Then $V=T\ot_AP$ is a $B$-$K$ bimodule with formal dual $\hV=\hP\ot_A\hT$.
\end{lemma}

\begin{proof}
We define $\ev_V$ out of $\ev_P$ and $\ev_T$ as follows
\begin{eqnarray*}
\ev_V=\ev_P\circ(\hP\ot_A\ev_T\ot_AT):\hP\ot_A\hT\ot_ST\ot_AP\to K.
\end{eqnarray*}
Similarly, we define $\hev_V$ out of $\hev_P$ and $\hev_T$ by
\begin{eqnarray*}
\hev_V=\hev_T\circ(T\ot_A\hev_P\ot_A\hT):T\ot_AP\ot_Z\hP\ot_A\hT\to B.
\end{eqnarray*}
An easy computation shows that the commutativity of the diagrams \eqref{diag.A} and \eqref{diag.B} applied on $P$ and $T$ forces the same diagrams for $V$ to be comutative.
\end{proof}

Let $\Ee$ and $\Cc$ be two $A$-corings and consider an $A$-bimodule map
$$\sigma:\Cc\ot_A\Ee\to \Ee\ot_A\Cc,$$
which renders commutative the following diagrams,
\begin{equation}\label{eq.sigmaE}
\hspace{-1.2cm}
\xymatrix{
& \Cc\!\ot_A\!\Cc\!\ot_A\!\Ee \ar[dl]_-{\Cc\ot_A\sigma} && \Cc\!\ot_A\!\Ee \ar[dd]_-{\sigma} \ar[rr]^-{\Cc\ot_A\Delta_\Ee} \ar[ll]_-{\Delta_\Cc\ot_A\Ee} && \Cc\!\ot_A\!\Ee\!\ot_A\!\Ee \ar[dr]^-{\sigma\ot_A\Ee}\\
\hspace{1cm}\Cc\!\ot_A\!\Ee\!\ot_A\!\Cc \ar[dr]_-{\sigma\ot_A\Cc} &&&&&& \hspace{-1cm}\Ee\!\ot_A\!\Cc\!\ot_A\!\Ee \ar[dl]^-{\Ee\ot_A\sigma}\\
&\Ee\!\ot_A\!\Cc\!\ot_A\!\Cc \ar[rr]_-{\Ee\ot_A\Delta_\Cc}&& \Ee\!\ot_A\!\Cc\ar[rr]_-{\Delta_\Ee\ot_A\Cc} && \Ee\!\ot_A\!\Ee\!\ot_A\!\Cc 
}
\end{equation}
\begin{equation}\label{eq.varepsilonE}
\xymatrix{
\Cc\ot_A\Ee \ar[d]_-{\sigma} \ar[rr]^-{\Cc\ot_A\varepsilon_\Ee} && \Cc\ot_AA \ar[d]^-{\cong}\\
\Ee\ot_A\Cc\ar[rr]_-{\varepsilon_\Ee\ot_A\Cc} && A\ot_A\Cc ,
}\qquad
\xymatrix{
\Cc\ot_A\Ee \ar[d]_-{\sigma} \ar[rr]^-{\varepsilon_\Cc\ot_A\Ee} && A\ot_A\Ee \ar[d]^-{\cong}\\
\Ee\ot_A\Cc\ar[rr]_-{\Ee\ot_A\varepsilon_\Cc} && \Ee\ot_AA .
}
\end{equation}
These conditions hold if and only if the $A$-bimodule $\Ee\ot_A\Cc$ is $A$-coring with coproduct $(\Ee\ot_A\sigma\ot_A\Cc)\circ(\Delta_\Ee\ot_A\Delta_\Cc)$ and counit $\varepsilon_\Ee\ot_A\varepsilon_\Cc$ (see \cite{CMZ} for the case of a commutative base). In this case, the $A$-coring structure on $\Ee\ot_A\Cc$ is called the {\em smash coproduct} of $\Ee$ and $\Cc$ and denoted by $\Ee\ot_\sigma\Cc$.

\begin{theorem}\label{th:composition}
Let $(T,\gamma_T)$ be a $B$-$A$ herd  and $(P,\gamma_P)$ an $A$-$K$ herd. Denote by $\Cc=\hT\ot_ST$ the $A$-coring associated to $T$ as in Corollary~\ref{cor.coring} and $\Ee=P\ot_Z\hP$ the $A$-coring associated to $P$. 
Then $V=T\ot_AP$ is a $B$-$K$ herd with shepherd $\gamma_V$ satisfying 
\begin{eqnarray}\label{eq.torsorV}
\gamma_{T,A}\ot_AP=(T\ot_A\hev_P\ot_A\hT\ot_ST\ot_AP)\circ\gamma_V\, ,\\
T\ot_A\gamma_{P,A}=(T\ot_AP\ot_Z\hP\ot_A\ev_T\ot_AP)\circ\gamma_V\, ,\nonumber
\end{eqnarray}
if and only if there exists a map
$\sigma:\Cc\ot_A \Ee\to \Ee\ot_A \Cc$, which defines a smash coproduct $\Ee\ot_\sigma \Cc$. (Hence $\Cc\ot_A \Ee$ is an $A$-coring).
\end{theorem}

\begin{proof}
By Lemma~\ref{lemma.composition.dual}, $\hV=\hP\ot_A\hT$ is a formal dual of $T\ot_AP$. 
Suppose first that $\Cc\ot_\sigma\Ee$ is a smash coproduct. Denote by $\gamma_{T,A}:T\to T\ot_A\hT\ot_ST$ and $\gamma_{P,A}:P\ot_Z\hP\ot_AP$ the projections of $\gamma_T$ and $\gamma_P$ respectively, constructed as in Notation \ref{not.gamma}. Define $\gamma_V:V\to V\ot_K\hV\ot_SV$ as the following composition: 
\[
\xymatrix{
V=T\ot_AP \ar[d]_-{\gamma_{T,A}\ot_A\gamma_{P,A}} \ar[rr]^-{\gamma_V} && V\ot_K\hV\ot_SV \ar@{=}[d] \\
T\ot_A\hT\ot_ST\ot_A P\ot_Z\hP\ot_AP \ar[rr]^-{T\ot_A\sigma\ot_AP}  
&& T\ot_AP\ot_Z\hP\ot_A\hT\ot_ST\ot_A P 
}
\]
We need to check that $\gamma_V$ satisfies diagrams \eqref{eq.coev}--\eqref{eq.ass}. First note that $\hev_P=\varepsilon_\Ee$. Diagram \eqref{eq.coev} for $V$ then comes out as (unadorned tensor product is over $A$) 

{\small
\[
\xymatrix{
T\!\otA \! P\ar[dd]^{\cong} \ar[rr]^-{\gamma_{T,A}\otA \gamma_{P,A}}  
&& T\! \otA \! \hT\! \ot_S\! T\! \otA \! P\! \ot_Z\! \hP\! \otA \!  P \ar[rr]^-{T \otA \sigma\otA P } \ar[d]_-{T\otA \hT\ot_ST\otA \hev_P\otA  P}
&& T\! \otA \! P\! \ot_Z\! \hP\! \otA \! \hT\! \ot_S\! T\! \otA \!  P \ar[d]_{T\otA \hev_P\otA \hT\ot_ST\otA P}\\
&& T\! \otA \! \hT\! \ot_S\! T\! \otA \! A\! \otA \!  P \ar[rr]^-\cong&& T\! \otA \! \hT\! \ot_S\! T\! \otA \!  P \ar[d]_{\hev_T\ot_ST\otA P}\\
R\! \ot_R\! T\! \otA \! P \ar[rrrr]_-{\beta\ot_ST\otA P}&& && B\! \ot_S\! T\! \otA \!  P \ .
}
\]
}

\noindent
The small square in this diagram commutes because of the left diagram in \eqref{eq.varepsilonE}, the other part of the diagram commutes by diagram \eqref{eq.coev} applied to $T$ and $P$. In the same way, one proves that $V$ satisfies the condition of diagram \eqref{eq.ev}. Diagram \eqref{eq.ass} for $V$ looks as follows (unadorned tensor product is over $A$)

{\small
\[
\xymatrix{
T\! \otA\! P \ar[rrr]^-{\gamma_{T,A}\otA\gamma_{P,A}} \ar[d]^-{\gamma_{T,A}\otA\gamma_{P,A}} &&& T\! \otA\! \Cc\! \otA\!  \Ee\! \otA\! P \ar[d]_{\gamma_{T,A}\otA\Cc\otA\gamma_{P,A}\ot_Z\hP\otA P}
\ar[rr]^-{T\otA\sigma\otA P} && { T\! \otA\! \Ee\! \otA\!  \Cc\! \otA \! P} \ar[d]_-{\gamma_{T,A}\otA \gamma_{P,A}\ot_Z \hP\otA \Cc\otA P} \\
T\! \otA\!  \Cc\! \otA \!  \Ee\! \otA\!  P \ar[d]^-{T\otA \sigma\otA P} \ar[rrr]^-{T\otA \hT\ot_S\gamma_{T,A}\otA \Ee \otA \gamma_{P,A}}
&&& T\! \otA\!  \Cc\! \otA\!  \Cc\! \otA\!  \Ee\! \otA\!  \Ee\! \otA\!  P \ar[d]_{T\otA \sigma^2\otA \Ee\otA P} \ar[rr]^-{T\otA \Cc\otA \sigma_2\otA P}
&& T\! \otA\!  \Cc\! \otA\!   \Ee\! \otA\!  \Ee \otA\!  \Cc\! \otA\!  P \ar[d]_{T\otA \sigma\otA \Ee\otA \Cc\otA P}\\
T\! \otA \! \Ee\! \otA\!  \Cc\! \otA\!  P \ar[rrr]_-{T\otA \Ee\otA \hT\ot_S\gamma_{T,A}\otA \gamma_{P,A}}
&&& T\! \ot_R\! \Ee\! \otA\!  \Cc\! \otA\!  \Cc\! \otA\!  \Ee\! \otA\!  P \ar[rr]_-{T\otA \Ee\otA \Cc\otA \sigma\otA P}
&& T\! \otA\!  \Ee\! \otA\!  \Cc\! \otA\!  \Ee\! \otA\!  \Cc\! \otA\!  P\ ,
}
\]
}

\noindent where  $\sigma^2=(\sigma\ot_A\Cc)\circ(\Cc\ot_A\sigma)$ and $\sigma_2=(\Ee\ot_A\sigma)\circ(\sigma\ot_A\Ee)$. The upper left square in this diagram commutes by \eqref{eq.ass} for $T$ and $P$, the upper right diagram commutes by the right pentagon in diagram \eqref{eq.sigmaE}, the lower left square commutes by the left  pentagon in \eqref{eq.sigmaE}, and the lower right square commutes trivially.
Finally, the equations \eqref{eq.torsorV} can be easily verified.

Conversely, suppose that $V=T\ot_AP$ is a herd with the shepherd $\gamma_V$. Then we define 
$\sigma:\Cc\ot_A\Ee\to \Ee\ot_A\Cc$ as follows,
\[
\xymatrix{
\Cc\!\ot_A\!\Ee \ar@{=}[d] \ar[rr]^-{\sigma} && \Ee\!\ot_A\!\Cc \ar@{=}[d] \\
\hT\!\ot_S\! T\!\ot_A\! P\! \ot_Z\! \hP \ar[dr]^{\hT\ot_S\gamma_V\ot_Z\hP} &&P\! \ot_Z\! \hP\! \ot_A\! \hT\! \ot_S\! T\\
&\hT\! \ot_S\! T\! \ot_A\! P\! \ot_Z\! \hP\! \ot_A\! \hT\! \ot_S\! T\! \ot_A\! P\! \ot_Z\! \hP \ar[ur]_{~\ev_T\ot_A\Ee\ot_A\Cc\ot_A\hev_P}
}
\]
By similar diagram chasing arguments, one proves that $\sigma$ is indeed defining a smash coproduct on $\Ee\ot_A\Cc$, provided that the equations \eqref{eq.torsorV} are satisfied.
\end{proof}

\section{Galois co-objects}\label{se:co-object}
\setcounter{equation}0
The aim of this section is to show how Galois co-objects for a commutative Hopf algebra and their composition can be interpreted in terms of bimodule herds. In this section we fix a commutative ground ring $k$, and do not deal with $k$-rings and $k$-corings, but with $k$-algebras and $k$-coalgebras. Troughout this section $H$ is a Hopf algebra (with coproduct $\Delta_H$, counit $\eps_H$ and the unit map $\eta_H: k\to H$, $x\mapsto x1_H$)  that is faithfully flat over its commutative base ring $k$, with a bijective antipode. The symbol $S$ denotes the antipode of a Hopf algebra $H$. The unadorned tensor product is over $k$. We use the Sweedler notation for coproduct, i.e.\ $\Delta_H(h) = h\sw 1\ot h\sw 2$, $\Delta^2_H(h) = h\sw 1\ot h\sw 2\ot h\sw 3$, etc.

\subsubsection*{Galois co-objects as Galois comodules}

Let $C$ be a right $H$-module coalgebra, that is, $C$ is a $k$-coalgebra, with coproduct $\Delta_C$ and counit $\eps_C$, and a right $H$-module such that, for all $c\in C$ and $h\in H$,
$$
\Delta_C(ch)=c\sw 1 h\sw 1\ot c\sw 2 h\sw 2, \qquad \eps_C(ch) = \eps_C(c)\eps_H(h).
$$
A right $(H,C)$-Hopf module $M$ is a right $k$-module that has a right $H$-module structure and a right $C$-comodule structure $\varrho^M: M\to M\otimes C$ with the following compatibility condition
$$
\varrho^M(mh)=m\sco 0 h\sw 1\ot m\sco 1 h\sw 2,
$$
for all $m\in M$ and $h\in H$, where  $\varrho^M(m) = m\sco 0\ot m\sco 1$ is the Sweedler notation for a coaction. 
The category of all $(H,C)$-Hopf modules with $H$-linear $C$-colinear maps between them is denoted by $\M^C_H(H)$. It is known that out of these data one can construct an $H$-coring $\Cc=H\ot C$, with $H$-bimodule structure
$$
g(h\ot c)g'=ghg'\sw 1\ot cg'\sw 2,
$$
for all $h,g,g'\in H$ and $c\in C$, coproduct $H\ot \Delta_C$ and counit $H\ot\eps_C$. In this way the category of $(H,C)$-Hopf modules is isomorphic to the category of right $\cC$-comodules, $\M^C_H(H)\cong \M^\Cc$. Furthermore, $C$ is a right $(H,C)$-Hopf module with the regular $H$-module and $C$-comodule structures. Hence there is a functor
$$\Gg=-\ot C:\M_k\to \M_H^C(H).$$
This functor has both a left adjoint $\Ff$ and a right adjoint $\Hh$ given by
\begin{eqnarray*}
\Ff=-\ot_Hk:\M_H^C(H)&\to& \M_k,\\
\Hh=\Hom_H^C(C,-):\M_H^C(H)&\to& \M_k.
\end{eqnarray*}
The adjointness of $(\Ff,\Gg)$ follows by \cite[Proposition 8.7.1]{Cae:blue}, the adjointness of $(\Gg,\Hh)$ is a general Hom-tensor relation.

An $H$-module coalgebra $C$ is called a \emph{Hopf-Galois co-object} if and only if the pair $(\Ff,\Gg)$ is an inverse equivalence. The uniqueness of adjoints implies that $C$ is a Galois co-object if and only if $(\Gg,\Hh)$ is a pair of inverse equivalences, which means in particular that $C$ is a right Galois comodule for the $H$-coring $\Cc=H\ot C$. By \cite[3.4 and 3.7]{V:equiv},  every Galois co-object $C$ is therefore finitely generated and projective as a right $H$-module. Furthermore, Example \ref{ex.Galois} shows that out of this Galois co-object we can construct a $B$-$H$ herd, where $B=\Endd^C_H(C)$, which describes exactly the Galois properties of $C$ as a right $\Cc$-comodule (see Theorem~\ref{thm.torsor.galois}).

\subsubsection*{The group of Galois co-objects and the composition of herds}

Recall from \cite[Theorem 8.7.4]{Cae:blue} that if $C$ is a Galois co-object for $H$, then the map
$$
\delta:C\ot H\to C\ot C,\qquad \delta(c\ot h)=c\sw 1\ot c\sw 2 h,
$$
is bijective.
Define $\hC$ as the left $H$-module, which is isomorphic to $C$ as  a $k$-module and with $H$-action given by
$$h\leftact \hc= \hc S^{-1}(h).$$
\begin{lemma}\label{le:coGalev}
Let $C$ be a Galois co-object for $H$, then 
$$
\ev:\hC\ot C\to H,\qquad \ev = (\eps_C \ot H)\circ \delta^{-1}, 
$$
 is an $H$-bilinear map.
\end{lemma}

\begin{proof}
The map $\ev$ is the {\em cotranslation map}, and the $H$-bilinearity property is a dualisation of the $H$-bicolinearity property of the {\em translation map};  see \cite[Remark~3.4~(d),~(e)]{Sch:rep}. We include the direct proof for completeness. 

Since $\delta$ is bijective, we can write an element in $C\ot C$ uniquely as a finite sum of elements  of the form $d\sw 1 \ot d\sw 2 h$, where $d\in C$, $h\in H$. By definition,  $\ev(d\sw 1 \ot d\sw 2 h)=\eps_C(d)h$. The map $\delta$ is a right $H$-module map, hence the right $H$-linearity of $\ev$ is clear. The left $H$-linearity is proven as follows:
\begin{eqnarray*}
\ev\bigg(h'\leftact d\sw 1\ot d\sw 2 h\bigg)&=& \ev\bigg(d\sw 1 S^{-1}(h')\ot d\sw 2 h\bigg) \\
&=&\ev\bigg(d\sw 1 S^{-1}(h')\sw 1\ot d\sw 2\varepsilon_C\big(S^{-1}(h')\sw 2\big) h\bigg)\\
&=&\ev\bigg(d\sw 1 S^{-1}(h')\sw 1\ot d\sw 2S^{-1}(h')\sw 2S\big(S^{-1}(h')\sw 3\big) h\bigg)\\
&=&\varepsilon\big(dS^{-1}(h')\sw 1\big)S\big(S^{-1}(h')\sw 2\big)h= \varepsilon_C(d)h'h\\
&=&h'\ev\big(d\sw 1\ot d\sw 2 h\big),
\end{eqnarray*}
where we used the antipode property in the third equality.
\end{proof}

Recall that the antipode $S$ of a commutative Hopf algebra $H$ is always involutive,  that is $S$ is bijective and $S^{-1} = S$.
Furthermore, if $H$ is a commutative Hopf algebra, then  the set of Galois co-objects forms a group with the tensor product over $H$ as the composition. Our next aim is to show that this composition can be obtained from the composition of herds as described in Section~\ref{se:composition}. For this purpose, we need to associate to a Galois co-object $C$ a bimodule herd different from the one described in example \ref{ex.Galois}.
For the rest of the section we assume that $H$ is commutative as $k$-algebra. A right $H$-module coalgebra $C$ is now understood as an $H$-bimodule with the same left and right action, and $\hC$ is an $H$-bimodule with the same left and right action $\leftact$ defined in the preamble to Lemma~\ref{le:coGalev}.

\begin{theorem}\label{th:Galcoobj}
Let $C$ be a Galois co-object for a commutative Hopf algebra $H$. With notation as above, $\hC$ is a formal dual of the $H$-bimodule $C$, with
$\ev =  (\eps_C \ot H)\circ \delta^{-1}$ and $\hev = S\circ \ev$.
Furthermore, $C$ is an $H$-$H$ herd with shepherd
$$\Delta_C^2=(\Delta_C\ot C)\circ\Delta_C=(C\ot \Delta_C)\circ \Delta_C: C\to C\ot \hC\ot C,$$
and pen $\hC$.
\end{theorem}

\begin{proof}
The maps $\ev$ and $\hev$ are $H$-bilinear by Lemma~\ref{le:coGalev} and by the fact that $S^{-1} = S$. 
Since $\delta$ is bijective, there is an isomorphism
$$
\vartheta=(C\ot\delta)\circ(\delta\ot H):C\ot H\ot H\to C\ot C\ot C.
$$
Thus to check the commutativity of diagrams  \eqref{diag.A} and \eqref{diag.B} suffices it  to evaluate them on elements $\vartheta(c\ot h\ot h') = c\sw 1\ot c\sw 2 h\sw 1\ot c\sw 3 h\sw 2h'$, where $c\in C$ and $h,h'\in H$. For \eqref{diag.A},
\begin{eqnarray*}
c\sw 1\ev(c\sw 2h\sw 1\ot c\sw 3h\sw 2 h')  \!\!\! &=\!\!\! & c\sw 1\varepsilon_C(c\sw 2)\varepsilon_H(h)h'=c\varepsilon_H(h)h'\\
& =\!\!\!  &
 \varepsilon_C(c\sw 1)S(h\sw 1)c\sw 2h\sw 2 h'=\hev(c\sw 1\ot c\sw 2h\sw 1)c\sw 3h\sw 2 h'. 
\end{eqnarray*}
where we used the antipode property in the penultimate equality.
The commutativity of  \eqref{diag.B} can be checked in a similar way: 
\begin{eqnarray*}
\ev(c\sw 1\ot c\sw 2h\sw 1)c\sw 3\leftact h\sw 2 h' \!\!\!&=\!\!\!&\varepsilon_C(c\sw 1) h\sw 1\leftact c\sw 2h\sw 2h' = cS(h\sw 1)h\sw 2h'\\
&\!\!\!=\!\!\!& c\varepsilon_H(h)h' = c\varepsilon_H(h)S^2(h') =
c\sw 1\rightact \varepsilon_C(c\sw 2)\varepsilon_H(h)S(h')\\
&\!\!\!=\!\!\!& c\sw 1\rightact \hev(c\sw 2h\sw 1\ot c\sw 3h\sw 2 h').
\end{eqnarray*}

Finally, we need to check that $\Delta^2_C$ is a shepherd. Clearly, the map satisfies the coassociativity condition. Since $\delta(c\ot 1_H)=c\sw 1\ot c\sw 2$,  $\ev(\Delta_C(c))=\varepsilon(c)1_H=\hev(\Delta_C(c))$.  Hence diagrams \eqref{eq.coev} and \eqref{eq.ev} commute as a consequence of the counit condition of $C$.
\end{proof}

\begin{theorem}
Let $C$ be a Galois co-object for a commutative Hopf-algebra $H$. Consider $C$ as an $H$-$H$ herd with a pen  $\hC$ and shephard $\Delta^2_C$ as in Theorem~\ref{th:Galcoobj}. Then the coalgebra $E$ defined by the  equaliser 
\[
\xymatrix{
E\ar[rr]^-{e} && \hC\ot C \ar@<.5ex>[rrrr]^-{(\ev\ot\hC\ot C)\circ(\hC\ot \Delta^2_C)} \ar@<-.5ex>[rrrr]_-{\eta_H\ot\hC\ot C} &&&& H\ot\hC\ot C
}
\]
 is isomorphic to $C$.
 
Symmetrically, the coalgebra $F$ defined by the equaliser of $(C\ot\hC\ot\hev)\circ(\Delta^2_C\ot\hC)$ and $(C\ot\hC\ot\eta_H)$ is also isomorphic to $C$.

Consequently, if $C$ is a flat $k$-module (and hence the herd $C$ is tame),  rings $A$ and $B$ constructed from the coherd associated to $C$ in Theorem~\ref{reconstr} are isomorphic to $H$.
\end{theorem}
\begin{proof}
Set $\omega=(\ev\ot\hC\ot C)\circ(\hC\ot \Delta^2_C)-(\eta_H\ot\hC\ot C)$. For any $\delta(c\ot h)=c\sw 1\ot c\sw 2h\in\hC\ot C$, 
$$
\omega(c\sw 1\ot c\sw 2h)=h\sw 1\ot c\sw 1h\sw 2\ot c\sw 2h\sw 3-1_H\ot c\sw 1\ot c\sw 2h.
$$
This implies that
\begin{equation}\label{eq.eq}
\omega\circ\Delta_C=0, \qquad (\eps_H\ot C\ot C)\circ \omega \circ \delta = \Delta_C\circ\mu_{C,H} -\delta,
\end{equation}
where $\mu_{C,H} : C\ot H\to C$ is the multiplication of $H$ on $C$. By the first of equations \eqref{eq.eq} and the universal property of equalisers,  there is a map $\nu_C:C\to E$ such that  $\Delta_C=e\circ\nu_C$. Since  $\Delta_C$ is injective, so is $\nu_C$. The second of equations \eqref{eq.eq} implies that $e\circ\nu_C\circ\mu_{C,H}\circ\delta^{-1}\circ e$ is the identity map on $E$. Hence $\nu_C$ is surjective.

The statement about the coalgebra $F$ follows by symmetric arguments. The statement about the rings $A$ and $B$ follows by the fact that $\ol{C}\cong \hC\ot E\cong F\ot \hC\cong C\ot C$ as $k$-modules and then by Theorem~\ref{reconstr}.
\end{proof}

Consider two $H$-module coalgebras $C$ and $D$ over a commutative Hopf algebra$H$. Then $C\ot_HD$ is again an $H$-module coalgebra with the $H$-module structure given by
$$
(c\ot_H d)h=c\ot_H(dh)= (ch)\ot_H d, 
$$
and comultiplication
$$
\Delta_{C\ot_HD}(c\ot_Hd)=c\sw 1\ot_H d\sw 1\ot c\sw 2\ot_Hd\sw 2.
$$
Moreover, if $C$ and $D$ are Galois co-objects, then $C\ot_HD$ is again a Galois co-object (see \cite[Section 10.1]{Cae:blue}). In particular, the map
$$
\delta_{C\ot_HD}:C\!\ot_H\!D\!\ot\! H\to C\!\ot_H\! D\! \ot\!  C\! \ot_H\! D,\qquad c\ot_H d\ot h \mapsto c\sw 1\ot_H d\sw 1\ot c\sw 2\ot_Hd\sw 2h,
$$
is an isomorphism as the composite of isomorphisms
\begin{eqnarray*}
C\ot_HD\ot H&\cong& C\ot_HD\ot H\ot_HH\cong (C\ot H)\ot_{H\ot H}(D\ot H)\\
&\cong& (C\ot C)\ot_{H\ot H}(D\ot D)\cong (C\ot_HD)\ot (C\ot_HD),
\end{eqnarray*}
where $\delta_C\ot_{H\ot H}\delta_D$ is  the penultimate isomorphism.

There are two ways of constructing a formal dual $\widehat{C\ot_HD}$ of $C\ot_HD$: one as in Theorem~\ref{th:Galcoobj} (with evaluation maps denoted by $\ev$ and $\hev$) the other as in Lemma~\ref{lemma.composition.dual}. The latter construction gives a formal dual of the form $\widehat{D}\ot_H\widehat{C}$. The following lemma asserts that both constructions are mutually equivalent.

\begin{lemma}
With notation as above, the twist map
$$\tau:\widehat{C\ot_HD}\to \widehat{D}\ot_H\widehat{C},$$
is an isomorphism of $H$-modules. Furthermore, the following diagrams commute,
\[
\xymatrix{
\widehat{C\ot_HD}\ot C\ot_HD \ar[rrr]^-{\ev} \ar[d]_{\tau\ot C\ot_HD} &&& H \ar@{=}[d] \\
\widehat{D}\ot_H\widehat{C}\ot C\ot_HD \ar[rr]_-{\widehat{D}\ot_H\ev_C\ot_HD} && \widehat{D}\ot_H D \ar[r]_-{\ev_D} & H
}
\]
\[
\xymatrix{
C\ot_HD\ot \widehat{C\ot_HD} \ar[rrr]^-{\hev} \ar[d]_{C\ot_HD\ot\tau} &&& H \ar@{=}[d] \\
C\ot_HD\ot \widehat{D}\ot_H\widehat{C} \ar[rr]_-{C\ot_H\hev_D\ot_H\widehat{C}} && C\ot_H \widehat{C} \ar[r]_-{\hev_C} & H
}
\]
\end{lemma}

\begin{proof}
We only check the commutativity of the first diagram, the commutativity of the second diagram follows by similar arguments. 
By bijectivity of $\delta_{C\ot_HD}$,  an element of $\widehat{C\ot_HD}\ot C\ot_HD$ is a $k$-linear combination of $c\sw 1\ot_H d\sw 1\ot c\sw 2\ot_Hd\sw 2h$, with $c\in C$, $d\in D$ and $h\in H$.  Note that $\ev(c\sw 1\ot_H d\sw 1\ot c\sw 2\ot_Hd\sw 2h)=\varepsilon_C(c)\varepsilon_D(d)h$. On the other hand
\begin{eqnarray*}
\ev_D\circ(\widehat{D}\ot_H\ev_C\ot_HD)(\tau(c\sw 1\ot_H d\sw 1)\ot c\sw 2\ot_Hd\sw 2h)\\
&&\hspace{-7cm}=\ev_D(d\sw 1\ot_H \ev_C(c\sw 1\ot c\sw 2)d\sw 2h)\\
&&\hspace{-7cm}=\ev_D(d\sw 1\ot_H \varepsilon_C(c)d\sw 2h)
=\varepsilon_C(c)\varepsilon_D(d)h.
\end{eqnarray*}
This completes the proof.
\end{proof}

\begin{lemma}\label{le:smash}
Given $H$-Galois co-objects $C$ and $D$, write $\Delta_{C,H}^2$ for the projection of $\Delta_C^2$ to $C\ot_H\widehat{C}\ot C$ and $\Delta_{D,H}^2$ for the projection of $\Delta_D^2$ to $D\ot\widehat{D}\ot_H D$. Consider the $H$-coring $\Cc=\widehat{C}\ot C$ with comultiplication 
$$
\widehat{C}\ot \Delta^2_{C,H}:\widehat{C}\ot C\to\widehat{C}\ot C\ot_H\widehat{C}\ot C,
$$ and counit $\ev_C$, and the $H$-coring $\Dd=D\ot \widehat{D}$ with comultiplication
$$
\Delta^2_{D,H}\ot\widehat{D}:D\ot\widehat{D}\to D\ot\widehat{D}\ot_H D\ot\widehat{D},
$$
and counit $\hev_D$ (see Corollary~\ref{cor.coring}). Then the map
$\sigma:\Cc\ot_H\Dd\to \Dd\ot_H\Cc$, for all $x\in C, d\in D, h\in H$, given by 
$$
\sigma(c\sw 1\ot c\sw 2\ot_H d\sw 1\ot d\sw 2h)=d\sw 1\ot d\sw 2\ot_H c\sw 1\ot c\sw 2S(h),$$
defines a smash coproduct between $\Cc$ and $\Dd$.
\end{lemma}

\begin{proof}
Note that $\sigma$ is well-defined, since the combined isomorphism
\[
\xymatrix{
C\ot D\ot H\cong C\ot H\ot_H D\ot H \ar[rr]^-{\delta_C\ot_H\delta_D} && C\ot C\ot_H D\ot D,
}
\]
means that any element of $\cC\ot_H\cD$ is a $k$-linear combination of $c\sw 1\ot c\sw 2\ot_Hd\sw 1\ot d\sw 2h$, for $c\in C, d\in D, h\in H$. The involutivity of $S$ and the definition of $\sigma$ immediately imply that $\sigma$ is a right $H$-linear map.
For the left $H$-linearity, first note that the repeated application of the antipode and counit axioms yields, for all $c\in C, d\in D, g,h\in H$
\begin{eqnarray*}
g \leftact c\sw 1\ot c\sw 2\ot_Hd\sw 1\ot d\sw 2h &&\\
&&\hspace{-5.5cm}=c\sw 1S(g)\sw 1\ot c\sw 2S(g)\sw 2 \ot_H d\sw 1S(S(g)\sw 3)\sw 1\ot 
d\sw 2S(S(g)\sw 3)\sw 2 S(S(S(g)\sw 3)\sw 3)h.
\end{eqnarray*}
Applying $\sigma$ to the above equation, and using the properties of the antipode, including $S^{-1} =S$, we obtain
\begin{eqnarray*}
d\sw 1S(S(g)\sw 3)\sw 1\ot d\sw 2 S(S(g)\sw 3)\sw 2 \ot_H
c\sw 1S(g)\sw 1\ot c\sw 2S(g)\sw 2   S(S(g)\sw 3)\sw 3 S(h)\\
&&\hspace{-12cm}=d\sw 1 S^2(g\sw 1)\ot d\sw 2 S^2(g\sw 2)\ot_H
c\sw 1 S(g\sw 5) \ot c\sw 2 S(g\sw 4) S^2(g\sw 3) S(h)\\
&&\hspace{-12cm}=d\sw 1 g\sw 1\ot d\sw 2 g\sw 2\ot_H
c\sw 1 S(g\sw 4) \ot c\sw 2 \varepsilon_H(g\sw 3) S(h)\\ 
&&\hspace{-12cm}=d\sw 1 g\ot d\sw 2 \ot_H
c\sw 1  \ot c\sw 2 S(h). 
\end{eqnarray*}
This proves that $\sigma$ is also  left $H$-linear. 
To check the left pentagon in \eqref{eq.sigmaE}, for any
$c\sw 1\ot c\sw 2\ot_Hd\sw 1\ot d\sw 2h\in\Cc\ot_H\Dd$, compute
\begin{eqnarray*}
(\sigma\ot_H\Cc)\circ(\Cc\ot_H\sigma)(c\sw 1\ot c\sw 2\ot_Hc\sw 3\ot c\sw 4\ot_Hd\sw 1\ot d\sw 2h)\\
&&\hspace{-6cm}=d\sw 1\ot d\sw 2\ot_Hc\sw 1\ot c\sw 2\ot_Hc\sw 3\ot c\sw 4S(h).
\end{eqnarray*}
On the other hand
\begin{eqnarray*}
(\Dd\ot_H\Delta_\Cc)\circ\sigma(c\sw 1\ot c\sw 2\ot_Hd\sw 1\ot d\sw 2h)\\
&&\hspace{-6cm}=d\sw 1\ot d\sw 2\ot_Hc\sw 1\ot c\sw 2S(h)\sw 1\ot_Hc\sw 3S(h)\sw 2\ot c\sw 4S(h)\sw 3\\
&&\hspace{-6cm}=d\sw 1\ot d\sw 2\ot_Hc\sw 1\ot c\sw 2\ot_Hc\sw 3S(S(h)\sw 1)S(h)\sw 2\ot c\sw 4S(h)\sw 3\\
&&\hspace{-6cm}=d\sw 1\ot d\sw 2\ot_Hc\sw 1\ot c\sw 2\ot_Hc\sw 3\ot c\sw 4S(h).
\end{eqnarray*}
The commutativity of the right pentagon \eqref{eq.sigmaE} is easy. To check the right diagram in \eqref{eq.varepsilonE} take again $c\sw 1\ot c\sw 2\ot_Hd\sw 1\ot d\sw 2h\in\Cc\ot_H\Dd$, and compute
\begin{eqnarray*}
d\sw 1\ot d\sw 2\rightact\ev_C(c\sw 1\ot c\sw 2S(h))&=&d\sw 1\ot d\sw 2\rightact \varepsilon_C(c)S(h)\\
&=&\varepsilon_C(c)d\sw 1\ot d\sw 2 h = \ev_C(c\sw 1\ot c\sw 2)d\sw 1\ot d\sw 2h .
\end{eqnarray*}
Similarly,
$$
\hev_D(d\sw 1\ot d\sw 2)\leftact c\sw 1\ot c\sw 2S(h) =\varepsilon_D(d)c\sw 1\ot c\sw 2S(h)
=c\sw 1\ot c\sw 2\hev_D(d\sw 1\ot d\sw 2h),
$$
which expresses the commutativity of the left diagram in \eqref{eq.varepsilonE}. So we conclude that $\sigma$ is a smash coproduct map as required.
\end{proof}

\begin{theorem}
Let $H$ be a commutative Hopf algebra. Consider two Galois co-objects $C$ and $D$. 
Then the herd associated to the composed Galois co-object $C\ot_HD$ coincides with the herd obtained by composing the herd associated to $C$ and the herd associated to $D$, using the smash coproduct described in Lemma \ref{le:smash}.
\end{theorem}

\begin{proof}
By Theorem \ref{th:Galcoobj},  the shepherd of $C\ot_HD$ is given by 
$$
\Delta_{C\ot_HD}^2:C\ot_HD\to C\ot_HD\ot\widehat{C\ot_HD}\ot C\ot_HD.
$$ 
On the other hand, the shepherd of the composed herd  constructed by Theorem~\ref{th:composition} is given by 
$$
\gamma_{C\ot_HD}=\left(C\!\ot_H\!\sigma\!\ot_H\!D\right)\circ\left(\Delta^2_{C,H}\ot_H\Delta_{D,H}^2\right):C\!\ot_H\!D\to C\ot_HD\ot\widehat{D}\ot_H\widehat{C}\ot C\ot_HD.
$$ 
These two herd structures can be mutually identified by the commutativity of  the following diagram
\[
\xymatrix{
C\ot_HD \ar[rr]^-{\Delta_{C\ot_HD}^2} \ar[d]_{\Delta^2_{C,H}\ot_H\Delta^2_{D,H}} && C\ot_HD\ot\widehat{C\ot_HD}\ot C\ot_HD \ar[d]^{C\ot_HD\ot\tau\ot C\ot_HD} \\
C\ot_H\widehat{C}\ot C\ot_HD\ot\widehat{D}\ot_HD \ar[rr]_-{C\ot_H\sigma\ot_HD} 
&& C\ot_HD\ot\widehat{D}\ot_H\widehat{C}\ot C\ot_HD \ .
}
\]
\end{proof}

\begin{remarks}
Dualising the results of this section, it is possible to construct bicomodule coherds out of Galois objects over a cocommutative Hopf algebra $H$. These Galois objects are known to form a group under the cotensor product. This composition would be then   related to a composition of bicomodule coherds by means of smash products.

In \cite[Chapter 10]{Fem:PhD}, the group of Galois coobjects for a commutative Hopf-algebroid is computed. The results of this section can be extended to this more general framework. 
\end{remarks}

\appendix

\section{The categorical formulation of (co)herd bi(co)modules}
\setcounter{equation}0
Let $\fR$ and $\fS$ be categories, and take a monad $\mathbf{A} = (A,m_A,\eta_A)$ on $\fR$ ($A$ is an endofunctor: $\fR \to \fR$, $m_A$ is a multiplication and $\eta_A$ is a unit) and a monad $\mathbf{B} = (B,m_B,\eta_B)$ on 
$\fS$. Let $F: \fS\to \fR$ be an $A$-$B$ bialgebra (or bimodule) functor. This means that $F$ comes equipped with
natural transformations $\varrho: FB\to F$ and $\lambda: AF\to F$ such that
$$
\xymatrix{FBB \ar[rr]^-{Fm_B}\ar[d]_{\varrho B} && FB \ar[d]^{\varrho} \\
FB \ar[rr]^-{\varrho } && F\ ,} \qquad 
\xymatrix{FB \ar[r]^\varrho & F\\
 & F \ar[u]_{=} \ar[ul]^{F\eta_B} \ ,}
 $$
 $$
\xymatrix{AAF \ar[rr]^-{m_AF}\ar[d]_{A\lambda} && AF \ar[d]^{\lambda} \\
AF \ar[rr]^-{\lambda } && F\ ,} \qquad 
\xymatrix{FA \ar[r]^\lambda & F \\
 & F \ar[u]_{=} \ar[ul]^{F\eta_A} \ ,} \qquad 
 \xymatrix{AFB \ar[rr]^-{\lambda B}\ar[d]_{A\varrho} && FB \ar[d]^{\varrho} \\
AF \ar[rr]^-{\lambda  } && F\ .}
 $$
Consider  a $B$-$A$ bialgebra functor $\hF: \fR\to\fS$ with structure maps $\hrho$, $\hlambda$. Then $F\hF$ is an $A$-$A$ bialgebra functor with structure
natural transformations $\lambda \hF$ and $F\hrho$. Similarly $\hF F$ is a $B$-$B$ bialgebra functor with structure natural transformations $\hlambda F$ and $\hF \varrho$.  With these data $\hF$   together with bialgebra natural transformations
$$
\ev : F\hF\to A, \qquad \hev: \hF F\to B,
$$
such that
$$
 \xymatrix{F\hF F \ar[rr]^-{F\hev}\ar[d]_{\ev F} && FB \ar[d]^{\varrho} \\
AF \ar[rr]^-{\lambda } && F\ ,} \qquad
 \xymatrix{\hF F \hF \ar[rr]^-{\hev \hF}\ar[d]_{\hF \ev} && B\hF \ar[d]^{\hlambda} \\
\hF A\ar[rr]^-{\hrho } && \hF \ ,} 
$$
 is called a {\em formal dual} of $F$. An $A$-$B$ bialgebra $F$ together with a formal dual $\hF$ is called a {\em herd functor} if there exists a natural transformation, the {\em shepherd},  
$$
\gamma : F\to F\hF F,
$$
such that
$$
\xymatrix{ & F  \ar[dd]^\gamma \ar[dl]_{F\eta_B}\ar[dr]^{\eta_A F}& \\
FB & & AF  \ ,  \\
& F \hF F  \ar[ul]^{F\hev}\ar[ur]_{\ev F} &} \qquad
\xymatrix{F \ar[rr]^-{\gamma}\ar[dd]_{\gamma} && F\hF F\ar[dd]^{F\hF \gamma} \\ && \\
F\hF F \ar[rr]^-{\gamma F\hF} && {F\hF F \hF F}\ .} 
$$
Dually, one defines a {\em coherd functor} as a $C$-$D$ bicoalgebra  functor $F: \fS\to \fR$ 
of two comonads $C: \fR\to \fR$ and $D: \fS\to \fS$ with a companion $D$-$C$ bialgebra
functor 
$\bar{F} : \fR\to \fS$ together with a natural transformation $\chi: F\bar{F} F\to F$ satisfying axioms dual to the ones for a herd functor. 

Take $\fR = \fS = \mathbf{Set}$, fix a set $X$ and consider the $X$-representable 
functor $A=B=F=\hF = \Map(X,-)$. Since all the functors appearing in this example are representable, by the Yoneda lemma all the natural transformations between them are determined by suitable functions (elements of $\Map(Y,Z)$ for suitable sets $Y$ and $Z$). For example, the multiplication $m_A : \Map(X, \Map(X, -)) \cong \Map (X\times X, -) \to \Map (X,-)$ is determined by a mapping $\delta : X\to X\times X$ so that $m_A = \Map(\delta, -)$. Similarly, the unit is determined by the (only possible) function $X\to \{*\}$. The only choice for $\delta$ that makes $\Map(X,-)$ a monad is the diagonal mapping
$
\delta: x\mapsto (x,x).
$
Set
$$
\lambda = \varrho = \hlambda = \hrho = \ev =\hev = \Map(\delta, -).
$$
 A shepherd  $\gamma$ is determined by a function $\chi : X\times X\times X\to X$.
 In terms of the mapping $\chi$, the triangle and square diagrams for $\gamma$ read,
  for all $x_i \in X$, $i=1,\ldots, 5$,
$$
\chi(x_1,x_1,x_2) = x_2=\chi(x_2,x_1,x_1), 
$$
$$
\chi(\chi(x_1,x_2,x_3),x_4,x_5) = \chi(x_1,x_2,\chi(x_3,x_4,x_5)).
$$
Thus $\Map(X,-)$ is a herd functor on $\mathbf{Set}$ (with the formal dual and bicoalgebra structures described above) if and only if $X$ is a herd; see \cite[page 170]{Pru:the},  \cite[page 202, footnote]{Bae:ein}. This example justifies the choice of terminology.

Next take $\fR = \M_R$ and $\fS = \M_S$, the categories of right modules of rings $R$ and $S$ respectively, and  consider monads  $-\ot_R A$, $-\ot_S B$, for an $R$-ring $A$ and an $S$-ring $B$. Take $F$ to be the tensor functor $-\ot_S T$ (for an $S$-$R$ bimodule $T$). Then $F$ is a bialgebra of the above monads if and only if $T$ is a $B$-$A$ bimodule. Furthermore, a functor $\widehat{F}= -\ot_R\hT$ is a formal dual of $F$ if and only if $\hT$ is a formal dual of $T$. Finally, $T$ is a herd $B$-$A$ bimodule if and only if $-\ot_ST$ is a herd functor (with the formal dual, monads, etc.\ as specified above).

\section*{Acknowledgements} 
We are grateful to Zoran \v Skoda for inspiring comments. 
The first author would like to thank Gabriella B\"ohm for a discussion. 
The second author thanks the Fund for Scientific Research--Flanders
(Belgium) (F.W.O.--Vlaanderen) for a Postdoctoral Fellowship.

\end{document}